\documentclass[11pt]{amsart}
\usepackage{amssymb,amsfonts}
\usepackage{bbm}
\usepackage[english]{babel}
\usepackage[T1]{fontenc} 
\usepackage[utf8]{inputenc}  
\usepackage{enumitem}\allowdisplaybreaks

\usepackage{amsaddr}

\usepackage{vmargin}
\setmarginsrb{3cm}{2cm}{3cm}{2cm}{0cm}{2cm}{0cm}{1cm}

\usepackage{tikz}
\def\smath#1{\text{\scalebox{.7}{$#1$}}}
\definecolor{MyGreen}{rgb}{0.13,0.55,0.13}
\definecolor{vert}{rgb}{0.0, 0.5, 0.0}

\theoremstyle{plain}
\newtheorem{theorem}{Theorem}
\newtheorem{lemma}[theorem]{Lemma}
\newtheorem{corollary}[theorem]{Corollary}
\newtheorem{proposition}[theorem]{Proposition}
\theoremstyle{definition}

\newtheorem{example}[theorem]{Example}
\newtheorem{remark}[theorem]{Remark} 

\newcommand{\R}{\mathbb R}
\newcommand{\N}{\mathbb N}
\newcommand{\Z}{\mathbb Z}
\newcommand{\lex}{\mathrm{lex}}
\newcommand{\B}{\boldsymbol{\beta}}
\newcommand{\floor}[1]{\lfloor#1\rfloor}
\newcommand{\ceil}[1]{\lceil#1\rceil}
\newcommand{\Int}{[\![0,p-1]\!]}

\title{Dynamical behavior of alternate base expansions}

\author{\'Emilie Charlier$^1$,  C\'elia Cisternino$^{1,*}$ and Karma Dajani$^2$}
 \address{$^1$Department of Mathematics,
University of Li\`ege,
All\'ee de la D\'ecouverte 12,
4000 Li\`ege, Belgium\\
$^2$Department of Mathematics,
Utrecht University, P.O. Box 80010, 3508TA Utrecht, The Netherlands
}
\email{echarlier@uliege.be, ccisternino@uliege.be and k.dajani1@uu.nl}
\thanks{$^*$Corresponding author.}

\begin{document}

\begin{abstract}
We generalize the greedy and lazy $\beta$-transformations for a real base $\beta$ to the setting of alternate bases $\B=(\beta_0,\ldots,\beta_{p-1})$, which were recently introduced by the first and second authors as a particular case of Cantor bases. As in the real base case, these new transformations, denoted $T_{\B}$ and $L_{\B}$ respectively, can be iterated in order to generate the digits of the greedy and lazy $\B$-expansions of real numbers. The aim of this paper is to describe the dynamical behaviors of $T_{\B}$ and $L_{\B}$. We first prove the existence of a unique absolutely continuous (with respect to an extended Lebesgue measure, called the $p$-Lebesgue measure) $T_{\B}$-invariant measure. We then show that this unique measure is in fact equivalent to the $p$-Lebesgue measure and that the corresponding dynamical system is ergodic and has entropy $\frac{1}{p}\log(\beta_{p-1}\cdots \beta_0)$. We then express the density of this measure and compute the frequencies of letters in the greedy $\B$-expansions. We obtain the dynamical properties of $L_{\B}$ by showing that the lazy dynamical system is isomorphic to the greedy one. We also provide an isomorphism with a suitable extension of the $\beta$-shift. Finally, we show that the $\B$-expansions can be seen as $(\beta_{p-1}\cdots \beta_0)$-representations over general digit sets and we compare both frameworks. 
\end{abstract}

\maketitle

\bigskip
\hrule
\bigskip



\bigskip
\hrule
\bigskip

\section{Introduction}

A representation of a nonnegative real number $x$ in a real base $\beta>1$ is an infinite sequence $a_0a_1a_2\cdots$ of nonnegative integers such that $x=\sum_{i=0}^\infty \frac{a_i}{\beta^{i+1}}$. Among all $\beta$-representations, the greedy and lazy ones play a special role. They can be generated by iterating the so-called greedy $\beta$-transformations $T_\beta$ and lazy $\beta$-transformations $L_\beta$ respectively. The dynamical properties of $T_\beta$ and $L_\beta$ are now well understood since the seminal works of Rényi~\cite{Renyi:1957} and Parry~\cite{Parry:1960}; for example, see \cite{Dajani&Kraaikamp:2002-2}.

In a recent work, the first two authors introduced the notion of expansions of real numbers in a real Cantor base, that is, an infinite sequence of real bases $\B=(\beta_n)_{n\ge 0}$ satisfying $\prod_{n=0}^\infty\beta_n=\infty$~\cite{Charlier&Cisternino:2020}. In this initial work, the focus was on the combinatorial properties of these expansions. In particular, generalizations of several combinatorial results of real base expansions were obtained, such as Parry's criterion for greedy $\beta$-expansions or Bertrand-Mathis characterization of sofic $\beta$-shifts. The latter result was obtained for the subclass of periodic Cantor bases, namely the alternate bases.

The aim of this paper is to study the dynamical behaviors of the greedy and lazy expansions in an alternate base $\B=(\beta_0,\ldots,\beta_{p-1},\beta_0,\ldots,\beta_{p-1},\ldots)$. It is organized as follows. In Section~\ref{Section : Preliminaries}, we provide the necessary background on measure theory and on expansions of real numbers in a real base. In Section~\ref{Section : AlternateBases}, we introduce the greedy and lazy alternate base expansions and define the associated transformations $T_{\B}$ and $L_{\B}$. Section~\ref{Section : InvariantMeasure} is concerned with the dynamical properties of the greedy transformation. We first prove the existence of a unique absolutely continuous (with respect to an extended Lebesgue measure, called the $p$-Lebesgue measure) $T_{\B}$-invariant measure and then prove that this measure is equivalent to the $p$-Lebesgue measure and that the corresponding dynamical system is ergodic. We then express the density of this measure and compute the frequencies of letters in the greedy $\B$-expansions. In Section~\ref{Section : IsoGreedyLazy} and~\ref{Section : Shift}, we prove that the greedy dynamical system is isomorphic to the lazy one, as well as to a suitable extension of the $\beta$-shift. In Section~\ref{Section : ComparisonBexpansionProduct}, we show that the $\B$-expansions can be seen as $(\beta_{p-1}\cdots \beta_0)$-representations over general digit sets and we compare both frameworks.

\section{Preliminaries}
\label{Section : Preliminaries}
\subsection{Measure preserving dynamical systems}
A \emph{probability space} is a triplet $(X,\mathcal{F},\mu)$ where $X$ is a set, $\mathcal{F}$ is a $\sigma$-algebra over $X$ and $\mu$ is a measure on $\mathcal{F}$ such that $\mu(X)=1$. For a measurable transformation $T \colon X \to X$ and a measure $\mu$ on $\mathcal{F}$, the measure $\mu$ is \emph{$T$-invariant}, or equivalently, the transformation $T\colon X\to X$ is \emph{measure preserving with respect to $\mu$}, if for all $B \in \mathcal{F}$, $\mu(T^{-1}(B))=\mu(B)$. A \emph{dynamical system} is a quadruple $(X,\mathcal{F}, \mu,T)$ where $(X,\mathcal{F}, \mu)$ is a probability space and $T\colon X\to X$ is a measure preserving transformation with respect to $\mu$. A dynamical system $(X,\mathcal{F}, \mu,T)$ is \emph{ergodic} if for all $B\in \mathcal{F}$, $T^{-1}(B)=B$ implies $\mu(B)\in\{0,1\}$, and is \emph{exact} if $\bigcap_{n=0}^{\infty} \{T^{-n}(B)\colon B\in\mathcal{F}\}$ only contains sets of measure $0$ or $1$. Clearly, any exact dynamical system is ergodic. Two dynamical systems $(X,\mathcal{F}_X, \mu_X,T_X)$ and $(Y,\mathcal{F}_Y, \mu_Y,T_Y)$ are \emph{(measure preservingly) isomorphic} if there exists a $\mu_X$-a.e.\ injective measurable map $\psi\colon X\to Y$ such that $\mu_Y=\mu_X\circ \psi^{-1}$ and $\psi\circ T_X=T_Y\circ \psi$ $\mu_X$-a.e. 

For two measures $\mu$ and $\nu$ on the same $\sigma$-algebra $\mathcal{F}$, $\mu$ is \emph{absolutely continuous with respect to} $\nu$ if for all $B\in \mathcal{F}$, $\nu(B)=0$ implies $\mu(B)=0$, and $\mu$ and $\nu$ are \emph{equivalent} if they are absolutely continuous with respect to each other. In what follows, we will be concerned by the Borel $\sigma$-algebras $\mathcal{B}(A)$, where $A\subset\R$. In particular, a measure on $\mathcal{B}(A)$ is \emph{absolutely continuous} if it is absolutely continuous with respect to the Lebesgue measure $\lambda$ restricted to $\mathcal{B}(A)$. The Radon-Nikodym theorem states that $\mu$ and $\nu$ are two probability measures such that $\mu$ is absolutely continuous with respect to $\nu$, then there exists a $\nu$-integrable map $f \colon X\mapsto [0,+\infty)$ such that for all $B\in\mathcal{F}$, $\mu(B)=\int_B f\, d\nu$. Moreover, the map $f$ is $\nu$-a.e.\ unique. It is called the \emph{density} of the measure $\mu$ with respect to $\nu$ and is usually denoted $\frac{d\mu}{d\nu}$. 

For more details on measure theory and ergodic theory, we refer the reader to~\cite{Boyarsky&Gora:1997,Dajani&Kraaikamp:2002,Furstenberg:1981}.

\subsection{Real base expansions}
\label{Section : RealBases}

Let $\beta$ be a real number greater than $1$. A \emph{$\beta$-representation} of a nonnegative real number $x$ is an infinite sequence $a_0a_1a_2\cdots$ over $\N$ such that $x=\sum_{i=0}^\infty \frac{a_i}{\beta^{i+1}}$. For $x\in [0,1)$, a particular $\beta$-representation of $x$, called the \emph{greedy $\beta$-expansion of $x$}, is obtained by using the \emph{greedy algorithm}. If the first $N$ digits of the $\beta$-expansion of $x$ are given by $a_0,\ldots,a_{N-1}$, then the next digit $a_N$ is the greatest integer in $[\![0,\ceil{\beta}-1]\!]$ such that 
\[
	\sum_{n=0}^N\frac{a_n}{\beta^{n+1}}\le x.
\]
The greedy $\beta$-expansion can also be obtained by iterating the \emph{greedy $\beta$-transformation} 
\[
	T_{\beta} \colon [0,1) \to [0,1),\ 
	x \mapsto \beta x - \floor{\beta x}
\]
by setting $a_n=\floor{\beta T_{\beta}^n(x)}$ for all $n\in\N$. 

\begin{example}
\label{Ex : Tbeta}
In this example and throughout the paper, $\varphi$ designates the golden ratio, i.e., $\varphi=\frac{1+\sqrt{5}}{2}$. The transformation $T_{\varphi^2}$ is depicted in Figure~\ref{Fig : Tbeta}.
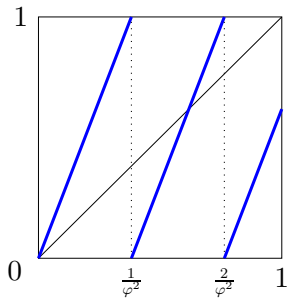
\begin{figure}[htb]
\centering
\begin{tikzpicture}[scale=0.8]
\draw[line width=0.0mm ] (0,0) rectangle (4,4); 
\draw[line width=0.0mm ] (0,0)  to (4,4); 
\draw (-0.1,-0.2) node[left]{$0$};
\draw (4,0) node[below]{$1$};
\draw (0,4) node[left]{$1$};
\draw[line width=0.4mm,blue] (0,0)  to (1.52786,4);
\draw[dotted] (1.52786,0) to (1.52786,4); 
\draw (1.52786,0) node[below]{$\smath{\tfrac{1}{\varphi^2}}$};
\draw[line width=0.4mm,blue](1.52786,0)  to (3.05573,4);
\draw[dotted] (3.05573,0) to (3.05573,4); 
\draw (3.05573,0) node[below]{$\smath{\tfrac{2}{\varphi^2}}$};
\draw[line width=0.4mm,blue] (3.05573,0)  to (4,2.47212);
\end{tikzpicture}
\caption{The transformation $T_{\varphi^2}$.}
\label{Fig : Tbeta}
\end{figure}
\end{example}

Real base expansions have been studied through various points of view. We refer the reader to~ \cite[Chapter 7]{Lothaire:2002} for a survey on their combinatorial properties and \cite{Dajani&Kraaikamp:2002} for a survey on their dynamical properties. A fundamental dynamical result is the following. This summarizes results from~\cite{Parry:1960,Renyi:1957,Rohlin:1961}.

\begin{theorem}
\label{Thm : uniqueMeasure}
There exists a unique $T_\beta$-invariant absolutely continuous probability measure $\mu_\beta$ on $\mathcal{B}([0,1))$. Furthermore, the measure $\mu_\beta$ is equivalent to the Lebesgue measure on $\mathcal{B}([0,1))$ and the dynamical system $([0,1),\mathcal{B}([0,1)),\mu_{\beta},T_{\beta})$ is ergodic and has entropy $\log(\beta)$.
\end{theorem}

\begin{remark}
\label{Rem : NonSing}
It follows from Theorem~\ref{Thm : uniqueMeasure} that $T_{\beta}$ is \emph{non-singular with respect to the Lebesgue measure}, i.e., for all $B\in \mathcal{B}([0,1))$, $\lambda(B) = 0$ if and only if $\lambda(T_{\beta}^{-1}(B))=0$.  
\end{remark}

In what follows, we let 
\[
	x_\beta=\frac{\ceil{\beta}-1}{\beta-1}.
\] 
This value corresponds to the greatest real number that has a $\beta$-representation over the alphabet $[\![0,\ceil{\beta}-1]\!]$. Clearly, we have $x_\beta\ge 1$. The \emph{extended greedy $\beta$-transformation}, still denoted $T_\beta$, is defined in~\cite{Dajani&Kraaikamp:2002-2} as
\[
	T_{\beta}\colon [0,x_\beta) \to [0,x_\beta),\ 
	x\mapsto 
	\begin{cases}
	\beta x-\floor{\beta x}
	&\text{if } x\in[0,1)\\
	\beta x -(\ceil{\beta}-1) 	
	&\text{if }x\in[1,x_\beta).
\end{cases}
\]
Note that for all $x\in \big[\frac{\ceil{\beta}-1}{\beta},\frac{\ceil{\beta}}{\beta}\big)$, the two cases of the definition coincide since $\floor{\beta x}=\ceil{\beta}-1$. The extended $\beta$-transformation restricted to the interval $[0,1)$ gives back the classical greedy $\beta$-transformation defined above. Moreover, for all $x\in [0,x_\beta)$, there exists $N\in \N$ such that for all $n\ge N$, $T_\beta^n(x) \in [0,1)$. 

\begin{example}
We continue Example~\ref{Ex : Tbeta}. The extended greedy transformation $T_{\varphi^2}$ is depicted in Figure~\ref{Fig : TbetaExt}.
\begin{figure}[htb]
\centering
\begin{tikzpicture}[scale=1]
\draw[line width=0.0mm ] (0,0) rectangle (4.94427,4.94427); 
\draw[dotted, line width=0.3mm ] (0,4) to (4,4); 
\draw[dotted, line width=0.3mm ] (4,0) to (4,4); 
\draw[line width=0.0mm ] (0,0)  to (4.94427,4.94427); 
\draw (-0.1,-0.2) node[left]{$0$};
\draw (4,0) node[below]{$1$};
\draw (0,4) node[left]{$1$};
\draw[line width=0.4mm,blue] (0,0)  to (1.52786,4);
\draw[dotted] (1.52786,0) to (1.52786,4); 
\draw (1.52786,0) node[below]{$\smath{\tfrac{1}{\varphi^2}}$};
\draw[line width=0.4mm,blue](1.52786,0)  to (3.05573,4);
\draw[dotted] (3.05573,0) to (3.05573,4); 
\draw (3.05573,0) node[below]{$\smath{\tfrac{2}{\varphi^2}}$};
\draw[line width=0.4mm,blue] (3.05573,0)  to (4.94427,4.94427);
\draw (4.94427,0) node[below]{$\smath{\tfrac{2}{\varphi^2-1}}$};
\draw (0,4.94427) node[left] {$\smath{\tfrac{2}{\varphi^2-1}}$};
\end{tikzpicture}
\caption{The extended transformation $T_{\varphi^2}$.}
\label{Fig : TbetaExt}
\end{figure}
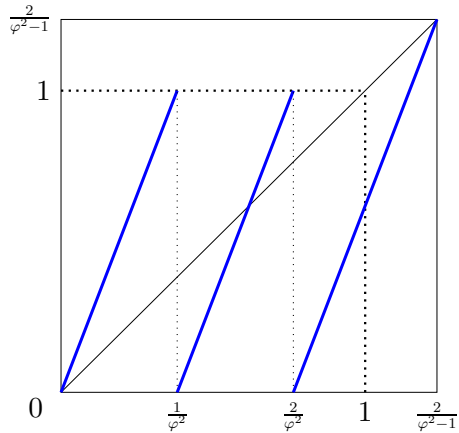
\end{example}

In the greedy algorithm, each digit is chosen as the largest possible among $0,1,\ldots,\ceil{\beta}-1$ at the considered position. At the other extreme, the \emph{lazy algorithm} picks the least possible digit at each step~\cite{Erdos&Joo&Komornik:1990}: if the first $N$ digits of the expansion of a real number $x\in(0,x_\beta]$ are given by $a_0,\ldots,a_{N-1}$, then the next digit $a_N$ is the least element in $[\![0,\ceil{\beta}-1]\!]$ such that 
\[
\sum_{n=0}^N\frac{a_n}{\beta^{n+1}}+\sum_{n=N+1}^{\infty}\frac{\ceil{\beta}-1}{\beta^{n+1}}\ge x,
\]
or equivalently,
\[
\sum_{n=0}^N\frac{a_n}{\beta^{n+1}}+\frac{x_\beta}{\beta^{N+1}}\ge x.
\]
The so-obtained $\beta$-representation is called the \emph{lazy $\beta$-expansion} of $x$. The \emph{lazy $\beta$-transfor\-mation} dynamically generating the lazy $\beta$-expansion is the transformation $L_{\beta}$ defined as follows~\cite{Dajani&Kraaikamp:2002}:
\[
	L_{\beta}\colon (0,x_\beta] \to (0,x_\beta],\ 
	x\mapsto 
	\begin{cases}
	\beta x &\text{if } x\in (0,x_\beta-1]\\
	\beta x -\ceil{\beta x -x_\beta} &\text{if }x\in (x_\beta-1,x_\beta].
\end{cases}
\]
Observe that for all $x\in
\big(\frac{x_\beta-1}{\beta},\frac{x_\beta}{\beta}\big]$, the two cases of the definition coincide since $\ceil{\beta x -x_\beta}=0$. Moreover, since $L_\beta\big((x_\beta-1,x_\beta]\big)=(x_\beta-1,x_\beta]$, the lazy transformation $L_\beta$ can be restricted to the length-one interval $(x_\beta-1,x_\beta]$. Also note that for all $x\in (0,x_\beta]$, there exists $N\in \N$ such that for all $n\ge N$, $L_\beta^n(x)\in(x_\beta-1,x_\beta]$. Furthermore, for all $x\in (x_\beta-1,x_\beta]$ and $n\in\N$, we have $a_n=\ceil{\beta L_{\beta}^n(x)-x_\beta}$.

\begin{example}
The lazy transformation $L_{\varphi^2}$ is depicted in Figure~\ref{Fig : Lbeta}.
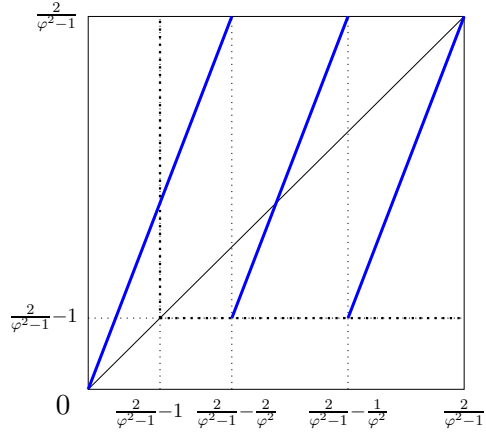
\begin{figure}[htb]
\centering
\begin{tikzpicture}[scale=1]
\draw[line width=0.0mm ] (0,0) rectangle (4.94427,4.94427); 
\draw[dotted, line width=0.3mm ]  (0.944272,0.944272) to (0.944272,4.94427); 
\draw[dotted, line width=0.3mm ]  (0.944272,0.944272)  to  (4.94427,0.944272); 
\draw[line width=0.0mm ] (0,0)  to (4.94427,4.94427); 
\draw[dotted]  (0.944272,0) to (0.944272,4.94427); 
\draw[dotted]  (0,0.944272)  to  (4.94427,0.944272); 
\draw (0,0.944272) node[left] {$\smath{\tfrac{2}{\varphi^2-1}{-}1}$};
\draw (0.8,0) node[below] {$\smath{\tfrac{2}{\varphi^2-1}{-}1}$};
\draw (-0.1,-0.2) node[left]{$0$};
\draw (4.94427,0) node[below]{$\smath{\tfrac{2}{\varphi^2-1}}$};
\draw (0,4.94427) node[left] {$\smath{\tfrac{2}{\varphi^2-1}}$};
\draw[line width=0.4mm,blue] (0,0)  to (1.88854,4.94427);
\draw[dotted] (1.88854,0) to (1.88854,4.94427); 
\draw (1.95,0) node[below]{$\smath{\tfrac{2}{\varphi^2-1}{-}\tfrac{2}{\varphi^2}}$};
\draw[line width=0.4mm,blue](1.88854,0.944272)  to (3.41641,4.94427);
\draw[dotted] (3.41641,0) to (3.41641,4.94427); 
\draw (3.41641,0) node[below]{$\smath{\tfrac{2}{\varphi^2-1}{-}\tfrac{1}{\varphi^2}}$};
\draw[line width=0.4mm,blue] (3.41641,0.944272)  to (4.94427,4.94427);
\end{tikzpicture}
\caption{The transformation $L_{\varphi^2}$.}
\label{Fig : Lbeta}
\end{figure}
\end{example}

It is proven in~\cite{Dajani&Kraaikamp:2002-2} that there is an isomorphism between the greedy 
and the lazy $\beta$-transformations. As a direct consequence of this property, an analogue of Theorem~\ref{Thm : uniqueMeasure} is obtained for the lazy transformation restricted to the interval $(x_\beta-1,x_\beta]$.

\section{Alternate base expansions}
\label{Section : AlternateBases}

Let $p$ be a positive integer and $\B=(\beta_0,\ldots,\beta_{p-1})$ be a $p$-tuple of real numbers greater than $1$. Such a $p$-tuple $\B$ is called an \emph{alternate base} and $p$ is called its \emph{length}. A \emph{$\B$-representation} of a nonnegative real number $x$ is an infinite sequence $a_0a_1a_2\cdots$ over $\N$ such that  
\begin{align}
\label{Eq : valueAlternatBase}
x\  =	
	&	\qquad\qquad\frac{a_0}{\beta_0} 
	&+	& \qquad\quad\frac{a_1}{\beta_1\beta_0} 
	&+	\quad \cdots \quad 
	&+	\quad\frac{a_{p-1}}{\beta_{p-1}\cdots\beta_0} \\
	&+	\quad\frac{a_p}{\beta_0(\beta_{p-1}\cdots\beta_0)}
    &+	&\quad \frac{a_{p+1}}{\beta_1\beta_0(\beta_{p-1}\cdots\beta_0)}
 	&+	\quad \cdots \quad 
 	&+	\;\ \frac{a_{2p-1}}{(\beta_{p-1}\cdots\beta_0)^2} \nonumber \\
 	&+	\qquad \cdots  \nonumber 
\end{align} 	
We use the convention that for all $n\in \Z$, $\beta_n=\beta_{n \bmod p}$ and $\B^{(n)}=(\beta_n,\ldots,\beta_{n+p-1})$. Therefore, the equality~\eqref{Eq : valueAlternatBase} can be rewritten
\[
	x=\sum_{n=0}^{+\infty}\frac{a_n}{\prod_{k=0}^n\beta_k}.
\]
The alternate bases are particular cases of Cantor real bases, which were introduced and studied in~\cite{Charlier&Cisternino:2020}.

In this paper, our aim is to study the dynamics behind some distinguished representation in alternate bases, namely the greedy and lazy $\B$-expansions. First, we recall the notion of greedy $\B$-expansions defined in~\cite{Charlier&Cisternino:2020} and we introduce the greedy $\B$-transformation dynamically generating the digits of the greedy $\B$-expansions. Second, we introduce the notion of lazy $\B$-expansions and the corresponding lazy $\B$-transformation.

\subsection{The greedy $\B$-expansion}
\label{Sec : Greedy}

For $x\in[0,1)$, a distinguished $\B$-representation, called the \emph{greedy $\B$-expansion} of $x$, is obtained from the \emph{greedy algorithm}. If the first $N$ digits of the greedy $\B$-expansion of $x$ are given by $a_0,\ldots,a_{N-1}$, then the next digit $a_N$ is the greatest integer in $[\![0,\ceil{\beta_N}-1]\!]$ such that 
\[
	\sum_{n=0}^N	\frac{a_n}{\prod_{k=0}^n\beta_k}\le x.
\]
The greedy $\B$-expansion can also be obtained by alternating the $\beta_i$-transformations: for all $x\in[0,1)$ and $n\in\N$, $a_n=\floor{\beta_n\big(T_{\beta_{n-1}}\circ \cdots \circ T_{\beta_0}(x)\big)}$. The greedy $\B$-expansion of $x$ is denoted $d_{\B}(x)$. In particular, if $p=1$ then it corresponds to the usual greedy $\beta$-expansion as defined in Section~\ref{Section : RealBases}. 

\begin{example}
\label{Ex : TbetaRacine13}
Consider the alternate base $\B=(\frac{1+\sqrt{13}}{2}, \frac{5+\sqrt{13}}{6})$ already studied in~\cite{Charlier&Cisternino:2020}. The greedy $\B$-expansions are obtained by alternating the transformations $T_{\frac{1+\sqrt{13}}{2}}$ and $T_{ \frac{5+\sqrt{13}}{6}}$, which are both depicted in Figure~\ref{Fig : TbetaRacine13}. Moreover, in Figure~\ref{Fig : First5DigitsTbetaRacine13} we see the computation of the first five digits of the greedy $\B$-expansion of $\frac{1+\sqrt{5}}{5}$. 
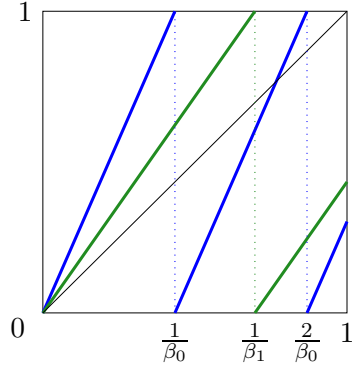
\begin{figure}[htb]
\begin{tikzpicture}
\draw[line width=0.0mm ] (0,0) rectangle (4,4); 
\draw (-0.1,-0.2) node[left]{$0$};
\draw (4,0) node[below]{$1$};
\draw (0,4) node[left]{$1$};
\draw (1.73703,0) node[below]{$\frac{1}{\beta_0}$};
\draw (3.47407,0) node[below]{$\frac{2}{\beta_0}$};
\draw[line width=0.4mm,blue] (0,0) to (1.73703,4); 
\draw[line width=0.4mm,blue] (1.73703,0)  to (3.47407,4); 
\draw[line width=0.4mm,blue] (3.47407,0)  to (4,1.2111); 
\draw (2.7889,0) node[below]{$\frac{1}{\beta_1}$};
\draw[line width=0.4mm,MyGreen] (0,0) to (2.7889,4); 
\draw[line width=0.4mm,MyGreen] (2.7889,0)  to (4,1.73703); 
\draw[line width=0.0mm ] (0,0)  to (4,4); 
\draw[blue,dotted] (1.73703,0)   to (1.73703,4) ; 
\draw[blue,dotted] (3.47407,0)   to (3.47407,4) ; 
\draw[MyGreen,dotted] (2.7889,0)   to (2.7889,4) ; 
\end{tikzpicture}
\caption{The transformations $T_{\frac{1+\sqrt{13}}{2}}$ (blue) and $T_{\frac{5+\sqrt{13}}{6}}$ (green).}
\label{Fig : TbetaRacine13}
\end{figure}
\begin{figure}[htb]
\begin{minipage}{.2\linewidth}
\begin{tikzpicture}[scale=0.6]
\draw[line width=0.0mm ] (0,0) rectangle (4,4); 
\draw[blue] (1.73703/2,4) node[above]{$0$};
\draw[blue] (2.60555,4) node[above]{$1$};
\draw[blue] (3.73704,4) node[above]{$2$};
\draw[line width=0.4mm,blue] (0,0) to (1.73703,4); 
\draw[line width=0.4mm,blue] (1.73703,0)  to (3.47407,4); 
\draw[line width=0.4mm,blue] (3.47407,0)  to (4,1.2111); 
\draw[line width=0.05mm,MyGreen] (0,0) to (2.7889,4); 
\draw[line width=0.05mm,MyGreen] (2.7889,0)  to (4,1.73703); 
\draw[line width=0.0mm ] (0,0)  to (4,4); 
\draw[line width=0.4mm,blue,dotted] (1.73703,0)   to (1.73703,4) ; 
\draw[line width=0.4mm,blue,dotted] (3.47407,0)   to (3.47407,4) ; 
\draw[MyGreen,dotted ] (2.7889,0)   to (2.7889,4) ; 
\filldraw[red] (2.58885,0) circle (3pt);
\draw[line width=0.4mm,red] (2.60555,4.42) circle [radius=0.3] node {$ $};
\draw[line width=0.4mm,blue] (2.60555,4) node[above]{$1$};
\draw[line width=0.4mm,red] (2.58885,0) to (2.58885,1.96155); 
\draw[line width=0.4mm,red] (1.96155,1.96155)to (2.58885,1.96155); 
\draw[line width=0.4mm,red] (1.96155,1.96155) to (1.96155,0);
\end{tikzpicture}
\end{minipage}%
\begin{minipage}{.2\linewidth}
\begin{tikzpicture}[scale=0.6]
\draw[line width=0.0mm ] (0,0) rectangle (4,4); 
\draw[MyGreen] (2.7889/2,4) node[above]{$0$};
\draw[MyGreen] (3.39445,4) node[above]{$1$};
\draw[line width=0.05mm,blue] (0,0) to (1.73703,4); 
\draw[line width=0.05mm,blue] (1.73703,0)  to (3.47407,4); 
\draw[line width=0.05mm,blue] (3.47407,0)  to (4,1.2111); 
\draw[line width=0.4mm,MyGreen] (0,0) to (2.7889,4); 
\draw[line width=0.4mm,MyGreen] (2.7889,0)  to (4,1.73703); 
\draw[line width=0.0mm ] (0,0)  to (4,4); 
\draw[blue,dotted] (1.73703,0)   to (1.73703,4) ; 
\draw[blue,dotted] (3.47407,0)   to (3.47407,4) ; 
\draw[line width=0.4mm,MyGreen,dotted ] (2.7889,0)   to (2.7889,4) ; 
\draw[line width=0.4mm,red] (2.7889/2,4.42) circle [radius=0.3] node {$ $};
\draw[MyGreen] (2.7889/2,4) node[above]{$0$};
\filldraw[red] (2.58885,0) circle (1.5pt);
\draw[line width=0.05mm,red] (2.58885,0) to (2.58885,1.96155); 
\draw[line width=0.05mm,red] (1.96155,1.96155)to (2.58885,1.96155); 
\draw[line width=0.05mm,red] (1.96155,1.96155) to (1.96155,0);
\draw[line width=0.4mm,red] (1.96155,2.81337) to (1.96155,0); 
\draw[line width=0.4mm,red] (1.96155,2.81337) to (2.81337,2.81337); 
\draw[line width=0.4mm,red] (2.81337,2.81337)to (2.81337,0); 
\filldraw[red] (1.96155,0)circle (3pt);
\end{tikzpicture}
\end{minipage}%
\begin{minipage}{.2\linewidth}
\begin{tikzpicture}[scale=0.6]
\draw[line width=0.0mm ] (0,0) rectangle (4,4); 
\draw[blue] (1.73703/2,4) node[above]{$0$};
\draw[blue] (2.60555,4) node[above]{$1$};
\draw[blue] (3.73704,4) node[above]{$2$};
\draw[line width=0.4mm,blue] (0,0) to (1.73703,4); 
\draw[line width=0.4mm,blue] (1.73703,0)  to (3.47407,4); 
\draw[line width=0.4mm,blue] (3.47407,0)  to (4,1.2111); 
\draw[line width=0.05mm,MyGreen] (0,0) to (2.7889,4); 
\draw[line width=0.05mm,MyGreen] (2.7889,0)  to (4,1.73703); 
\draw[line width=0.0mm ] (0,0)  to (4,4); 
\draw[line width=0.4mm,blue,dotted] (1.73703,0)   to (1.73703,4) ; 
\draw[line width=0.4mm,blue,dotted] (3.47407,0)   to (3.47407,4) ; 
\draw[MyGreen,dotted ] (2.7889,0)   to (2.7889,4) ; 
\filldraw[red] (2.58885,0) circle (1.5pt);
\draw[line width=0.4mm,red] (2.60555,4.42) circle [radius=0.3] node {$ $};
\draw[blue] (2.60555,4) node[above]{$1$};
\draw[line width=0.05mm,red] (2.58885,0) to (2.58885,1.96155); 
\draw[line width=0.05mm,red] (1.96155,1.96155)to (2.58885,1.96155); 
\draw[line width=0.05mm,red] (1.96155,1.96155) to (1.96155,0);
\draw[line width=0.05mm,red] (1.96155,2.81337) to (1.96155,0); 
\filldraw[red] (1.96155,0)circle (1.5pt);
\draw[line width=0.05mm,red] (1.96155,2.81337) to (2.81337,2.81337); 
\draw[line width=0.05mm,red] (2.81337,2.81337)to (2.81337,0); 
\draw[line width=0.4mm,red] (2.81337,2.47856)to (2.81337,0); 
\filldraw[red] (2.81337,0)circle (3pt);
\draw[line width=0.4mm,red] (2.81337,2.47856)to (2.47856,2.47856); 
\draw[line width=0.4mm,red] (2.47856,0)to (2.47856, 2.47856); 
\end{tikzpicture}
\end{minipage}%
\begin{minipage}{.2\linewidth}
\begin{tikzpicture}[scale=0.6]
\draw[line width=0.0mm ] (0,0) rectangle (4,4); 
\draw[MyGreen] (2.7889/2,4) node[above]{$0$};
\draw[MyGreen] (3.39445,4) node[above]{$1$};
\draw[line width=0.05mm,blue] (0,0) to (1.73703,4); 
\draw[line width=0.05mm,blue] (1.73703,0)  to (3.47407,4); 
\draw[line width=0.05mm,blue] (3.47407,0)  to (4,1.2111); 
\draw[line width=0.4mm,MyGreen] (0,0) to (2.7889,4); 
\draw[line width=0.4mm,MyGreen] (2.7889,0)  to (4,1.73703); 
\draw[line width=0.0mm ] (0,0)  to (4,4); 
\draw[blue,dotted] (1.73703,0)   to (1.73703,4) ; 
\draw[blue,dotted] (3.47407,0)   to (3.47407,4) ; 
\draw[line width=0.4mm,MyGreen,dotted ] (2.7889,0)   to (2.7889,4) ; 
\draw[line width=0.4mm,red] (2.7889/2,4.42) circle [radius=0.3] node {$ $};
\draw[MyGreen] (2.7889/2,4) node[above]{$0$};
\filldraw[red] (2.58885,0) circle (1.5pt);
\draw[line width=0.05mm,red] (2.58885,0) to (2.58885,1.96155); 
\draw[line width=0.05mm,red] (1.96155,1.96155)to (2.58885,1.96155); 
\draw[line width=0.05mm,red] (1.96155,1.96155) to (1.96155,0);
\filldraw[red] (1.96155,0)circle (1.5pt);
\draw[line width=0.05mm,red] (1.96155,2.81337) to (1.96155,0); 
\draw[line width=0.05mm,red] (1.96155,2.81337) to (2.81337,2.81337); 
\draw[line width=0.05mm,red] (2.81337,2.81337)to (2.81337,0); 
\draw[line width=0.05mm,red] (2.81337,2.47856)to (2.81337,0); 
\filldraw[red] (2.81337,0)circle (1.5pt);
\draw[line width=0.05mm,red] (2.81337,2.47856)to (2.47856,2.47856); 
\draw[line width=0.05mm,red] (2.47856,0)to (2.47856, 2.47856); 
\filldraw[red] (2.47856,0)circle (3pt);
\draw[line width=0.4mm,red] (2.47856,0)to (2.47856, 3.5549); 
\draw[line width=0.4mm,red] (2.47856, 3.5549) to(3.5549,3.5549) ; 
\draw[line width=0.4mm,red] (3.5549,3.5549) to (3.5549,0); 
\end{tikzpicture}
\end{minipage}%
\begin{minipage}{.2\linewidth}
\begin{tikzpicture}[scale=0.6]
\draw[line width=0.0mm ] (0,0) rectangle (4,4); 
\draw[blue] (1.73703/2,4) node[above]{$0$};
\draw[blue] (2.60555,4) node[above]{$1$};
\draw[blue] (3.73704,4) node[above]{$2$};
\draw[line width=0.4mm,blue] (0,0) to (1.73703,4); 
\draw[line width=0.4mm,blue] (1.73703,0)  to (3.47407,4); 
\draw[line width=0.4mm,blue] (3.47407,0)  to (4,1.2111); 
\draw[line width=0.05mm,MyGreen] (0,0) to (2.7889,4); 
\draw[line width=0.05mm,MyGreen] (2.7889,0)  to (4,1.73703); 
\draw[line width=0.0mm ] (0,0)  to (4,4); 
\draw[line width=0.4mm,blue,dotted] (1.73703,0)   to (1.73703,4) ; 
\draw[line width=0.4mm,blue,dotted] (3.47407,0)   to (3.47407,4) ; 
\draw[MyGreen,dotted ] (2.7889,0)   to (2.7889,4) ; 
\draw[line width=0.4mm,red] (3.73704,4.42) circle [radius=0.3] node {$ $};
\draw[blue] (3.73704,4) node[above]{$2$};
\filldraw[red] (2.58885,0) circle (1.5pt);
\draw[line width=0.05mm,red] (2.58885,0) to (2.58885,1.96155); 
\draw[line width=0.05mm,red] (1.96155,1.96155)to (2.58885,1.96155); 
\draw[line width=0.05mm,red] (1.96155,1.96155) to (1.96155,0);
\filldraw[red] (1.96155,0)circle (1.5pt);
\draw[line width=0.05mm,red] (1.96155,2.81337) to (1.96155,0); 
\draw[line width=0.05mm,red] (1.96155,2.81337) to (2.81337,2.81337); 
\draw[line width=0.05mm,red] (2.81337,2.81337)to (2.81337,0); 
\draw[line width=0.05mm,red] (2.81337,2.47856)to (2.81337,0); 
\filldraw[red] (2.81337,0)circle (1.5pt);
\draw[line width=0.05mm,red] (2.81337,2.47856)to (2.47856,2.47856); 
\draw[line width=0.05mm,red] (2.47856,0)to (2.47856, 2.47856); 
\filldraw[red] (2.47856,0)circle (1.5pt);
\draw[line width=0.05mm,red] (2.47856,0)to (2.47856, 3.5549); 
\draw[line width=0.05mm,red] (2.47856, 3.5549) to (3.5549,3.5549) ; 
\draw[line width=0.05mm,red] (3.5549,3.5549) to (3.5549,0); 
\filldraw[red] (3.5549,0) circle (3pt);
\draw[line width=0.4mm,red] (3.5549,0) to (3.5549,0.186135); 
\end{tikzpicture}
\end{minipage}
\caption{The first five digits of the greedy $\B$-expansion of $\frac{1+\sqrt{5}}{5}$ are $10102$ for $\B=(\frac{1+\sqrt{13}}{2}, \frac{5+\sqrt{13}}{6})$.}
\label{Fig : First5DigitsTbetaRacine13}
\end{figure}
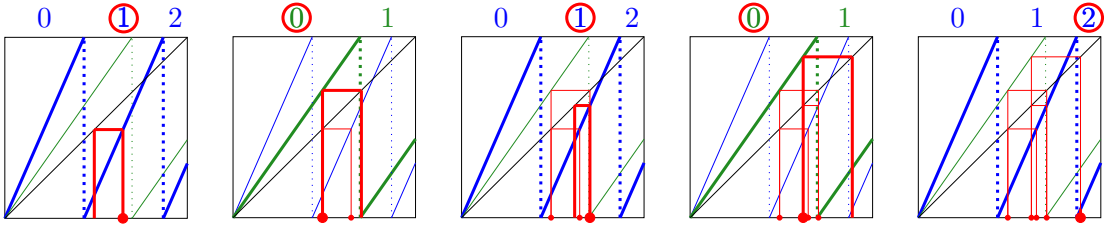
\end{example}

We now define the \emph{greedy $\B$-transformation} by
\begin{equation}
\label{Eq : TB}
	T_{\B}\colon 
	\Int\times [0,1) \to \Int\times [0,1),\ 
	(i,x) \mapsto \big((i+1)\bmod p, T_{\beta_i}(x)\big).
\end{equation}
The greedy $\B$-transformation generates the digits of the greedy $\B$-expansions as follows. For all $x\in[0,1)$ and $n\in \N$, the digit $a_n$ of $d_{\B}(x)$ is equal to $\floor{\beta_n\pi_2\big(T_{\B}^n(0,x)\big)}$ where 
\[
	\pi_2\colon\N\times \R\to \R,\ (n,x)\mapsto x. 
\]

As in Section~\ref{Section : RealBases}, the greedy $\B$-transformation can be extended to an interval of real numbers bigger than $[0,1)$. To do so, we define 
\begin{equation}
\label{Eq : x_B}
	x_{\B}=\sum_{n=0}^{\infty}\frac{\ceil{\beta_n}-1}{\prod_{k=0}^n\beta_k}.
\end{equation}
It can be easily seen that $1\le x_{\B}<\infty$. This value corresponds to the greatest real number that has a $\B$-representation $a_0a_1a_2\cdots$ such that each letter $a_n$ belongs to the alphabet $[\![0,\ceil{\beta_n}-1]\!]$. Moreover, for all $n\in\Z$, 
\begin{equation}
\label{Eq : EqualitiesXB}
	x_{\B^{(n)}}=\frac{x_{\B^{(n+1)}}+\ceil{\beta_n}-1}{\beta_n}.
\end{equation}
We define the \emph{extended greedy $\B$-transformation}, still denoted $T_{\B}$, by
\begin{align*}
T_{\B}\colon &\bigcup_{i=0}^{p-1} \big(\{i\} \times [0 , x_{\B^{(i)}} ) \big)\to 
\bigcup_{i=0}^{p-1} \big(\{i\} \times [0 , x_{\B^{(i)}}) \big),\\
 &(i,x) \mapsto 
 \begin{cases}
 \big((i+1)\bmod p,\beta_i x -\floor{\beta_i x}\big) 	
 & \text{if } x \in [0,1)\\
 \big((i+1) \bmod p, \beta_i x-(\ceil{\beta_i}-1)\big) 	
 & \text{if } x \in[1,x_{\B^{(i)}}).
 \end{cases}
\end{align*}
The greedy $\B$-expansion of $x\in [0,x_{\B})$ is obtained by alternating the $p$ maps 
\[
	{\pi_2 \circ T_{\B} \circ \delta_i}_{\big| [0,x_{\B^{(i)}})}
	\colon [0,x_{\B^{(i)}})\to [0,x_{\B^{(i+1)}})
\] 
for $i \in \Int$, where 
\[
	\delta_i\colon
	\R\to\{i\}\times\R,\ 
	x\mapsto (i,x).
\]

\begin{proposition}
\label{Pro : LexMax}
For all $x\in[0,x_{\B})$ and $n\in \N$, we have
\[
	\pi_2 \circ T_{\B}^n \circ \delta_0 (x)
	=\beta_{n-1}\cdots\beta_0 x
	-\sum_{k=0}^{n-1} \beta_{n-1}\cdots \beta_{k+1}c_k	
\]
where $(c_0,\ldots,c_{n-1})$ is the lexicographically greatest $n$-tuple in $\prod_{k=0}^{n-1}\ [\![0,\ceil{\beta_k}-1]\!]$ such that $\frac{\sum_{k=0}^{n-1} \beta_{n-1}\cdots \beta_{k+1}c_k}{\beta_{n-1}\cdots\beta_0} \le x$. 
\end{proposition}

\begin{proof}
We proceed by induction on $n$. The base case $n=0$ is immediate. Now, suppose that the result is satisfied for some $n\in\N$. Let $x\in[0,x_{\B})$. Let $(c_0,\ldots,c_{n-1})$ is the lexicographically greatest $n$-tuple in $\prod_{k=0}^{n-1}\ [\![0,\ceil{\beta_k}-1]\!]$  such that $\frac{\sum_{k=0}^{n-1} \beta_{n-1}\cdots \beta_{k+1}c_k}{\beta_{n-1}\cdots\beta_0} \le x$. Then it is easily seen that for all $m< n$, $(c_0,\ldots,c_m)$ is the lexicographically greatest $(m+1)$-tuple in $\prod_{k=0}^m\ [\![0,\ceil{\beta_k}-1]\!]$ such that $\frac{\sum_{k=0}^m \beta_m\cdots \beta_{k+1}c_k}{\beta_m\cdots\beta_0} \le x$. Now, set $y=\pi_2\circ T_{\B}^n \circ \delta_0 (x)$. Then $y\in[0,x_{\B^{(n)}})$ and by induction hypothesis, we obtain that $y=\beta_{n-1}\cdots\beta_0 x-\sum_{k=0}^{n-1} \beta_{n-1}\cdots \beta_{k+1}c_k$. Then, by setting  
\[
	c_n=
	\begin{cases}
	\floor{\beta_n y} &\text{if } y\in[0,1)\\
	\ceil{\beta_n}-1 &\text{if } y\in[1,x_{\B^{(n)}})
	\end{cases}
\]
we obtain that $\pi_2\circ T_{\B}^{n+1} \circ \delta_0 (x)
	=\beta_n\cdots \beta_0 x-\sum_{k=0}^n\beta_n\cdots \beta_{k +1}c_k$. In order to conclude, we have to show that 
\begin{enumerate}[label=\alph*)]
\item \label{Eq : a} $\frac{\sum_{k=0}^n\beta_n\cdots \beta_{k +1}c_k}{\beta_n\cdots \beta_0}\le x$
\item \label{Eq : b} $(c_0,\ldots,c_n)$ is the lexicographically greatest $(n+1)$-tuple in $\prod_{k=0}^n\ [\![0,\ceil{\beta_k}-1]\!]$ such that~\ref{Eq : a} holds.
\end{enumerate}

By definition of $c_n$, we have $c_n\le \beta_n y$. Therefore,
\[
	\sum_{k=0}^n\beta_n\cdots \beta_{k +1}c_k
	=\beta_n\sum_{k=0}^{n-1}\beta_{n-1}\cdots \beta_{k+1}c_k+c_n\\
	= \beta_n(\beta_{n-1}\cdots\beta_0x-y)+c_n\\
	\le\beta_n\cdots \beta_0x.
\]
This shows that~\ref{Eq : a} holds.

Let us show ~\ref{Eq : b} by contradiction. Suppose that there exists $(c_0',\ldots,c_n')\in \prod_{k=0}^n\ [\![0,\ceil{\beta_k}-1]\!]$ such that $ (c_0',\ldots,c_n')>_\lex(c_0,\ldots,c_n)$ and $\frac{\sum_{k=0}^n\beta_n\cdots \beta_{k +1}c_k'}{\beta_n\cdots \beta_0}\le  x$. Then there exists $m\le n$ such that $c_0'=c_0,\ldots,c_{m-1}'=c_{m-1}$ and $c_m'\ge c_m+1$. We again consider two cases. First, suppose that $m<n$. Since $(c_0',\ldots,c_m')>_\lex(c_0,\ldots,c_m)$, we get $\frac{\sum_{k=0}^m\beta_m\cdots \beta_{k+1}c_k'}{\beta_m\cdots \beta_0}> x$. But then 
\[
	\sum_{k=0}^n\beta_n\cdots \beta_{k+1}c_k'
\ge \beta_n\cdots \beta_{m+1}\sum_{k=0}^m\beta_m\cdots \beta_{k+1}c_k'
>\beta_n\cdots \beta_0 x,
\] 
a contradiction. Second, suppose that $m=n$. Then 
\[	
	\beta_n\cdots \beta_0 x
	\ge \sum_{k=0}^n\beta_n\cdots \beta_{k+1}c_k'\\
	\ge \sum_{k=0}^{n-1}\beta_n\cdots \beta_{k+1}c_k+c_n+1,
\]
hence $\beta_n y\ge c_n+1$. If $y\in[0,1)$ then $c_n+1=\floor{\beta_ny}+1>\beta_ny$, a contradiction. Otherwise, $y\in[1,x_{\B^{(n)}})$ and $c_n+1=\ceil{\beta_n}$. But then $c_n'\ge \ceil{\beta_n}$, which is impossible since $c_n'\in[\![0,\ceil{\beta_n}-1]\!]$. This shows~\ref{Eq : b} and ends the proof.
\end{proof}

The restriction of the extended greedy $\B$-transformation to the domain $\Int \times[0,1)$ gives back the greedy $\B$-transformation initially defined in~\ref{Eq : TB}. Moreover, for all $(i,x)\in\bigcup_{i=0}^{p-1}\big(\{i\}\times[0,x_{\B^{(i)}})\big)$, there exists $N \in \N$ such that for all $n\ge N$, $T_{\B}^n(i,x)\in \Int \times [0,1)$. 

\begin{example}
\label{Ex : TbetaRacine13SecondPart}
Let $\B=(\frac{1+\sqrt{13}}{2}, \frac{5+\sqrt{13}}{6})$ be the alternate base of Example~\ref{Ex : TbetaRacine13}. The maps ${\pi_2 \circ T_{\B} \circ \delta_0}_{\big| [0,x_{\B})}\colon  [0,x_{\B})\to  [0,x_{\B^{(1)}})$ and ${\pi_2 \circ T_{\B} \circ \delta_1}_{\big| [0,x_{\B^{(1)}})}\colon [0,x_{\B^{(1)}})\to  [0,x_{\B})$ are depicted in Figure~\ref{Fig : TbetaRacine13Extended}.
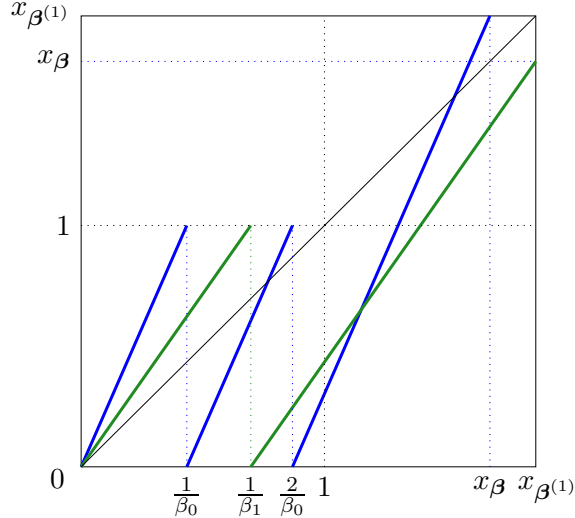
\begin{figure}[htb]
\begin{tikzpicture}[scale=0.8]
\draw[line width=0.0mm ] (0,0) rectangle (7.47408,7.47408); 
\draw (-0.1,-0.2) node[left]{$0$};
\draw (4,0) node[below]{$1$};
\draw (0,4) node[left]{$1$};
\draw (7.47408+0.2,0) node[below]{$x_{\B^{(1)}}$};
\draw (0,7.47408) node[left]{$x_{\B^{(1)}}$};
\draw (6.71976,0) node[below]{$x_{\B}$};
\draw (0,6.71976) node[left]{$x_{\B}$};
\draw (1.73703,0) node[below]{$\frac{1}{\beta_0}$};
\draw (3.47407,0) node[below]{$\frac{2}{\beta_0}$};
\draw[line width=0.4mm,blue] (0,0) to (1.73703,4); 
\draw[line width=0.4mm,blue] (1.73703,0)  to (3.47407,4); 
\draw[line width=0.4mm,blue] (3.47407,0)  to (6.71976,7.47408); 
\draw (2.7889,0) node[below]{$\frac{1}{\beta_1}$};
\draw[line width=0.4mm,MyGreen] (0,0) to (2.7889,4); 
\draw[line width=0.4mm,MyGreen] (2.7889,0)  to (7.47408,6.71976); 
\draw[blue,dotted] (6.71976,0)   to (6.71976,7.47408) ; 
\draw[blue,dotted] (0,6.71976)   to (7.47408,6.71976) ; 
\draw[line width=0.0mm ] (0,0)  to  (7.47408,7.47408); 
\draw[blue,dotted] (1.73703,0)   to (1.73703,4) ; 
\draw[blue,dotted] (3.47407,0)   to (3.47407,4) ; 
\draw[MyGreen,dotted] (2.7889,0)   to (2.7889,4) ; 
\draw[dotted] (0,4)   to (7.47408,4) ; 
\draw[dotted] (4,0)   to (4,7.47408) ; 
\end{tikzpicture}
\caption{
The maps ${\pi_2 \circ T_{\B} \circ \delta_0}_{\big| [0,x_{\B})}$ (blue) and ${\pi_2 \circ T_{\B} \circ \delta_1}_{\big| [0,x_{\B^{(1)}})}$ (green) with $\B=(\frac{1+\sqrt{13}}{2}, \frac{5+\sqrt{13}}{6})$.}
\label{Fig : TbetaRacine13Extended}
\end{figure}
\end{example}

\subsection{The lazy $\B$-expansion}
\label{Sec : Lazy}

As in the real base case, in the greedy $\B$-expansion, each digit is chosen as the largest possible at the considered position. Here, we define and study the other extreme $\B$-representation, called the \emph{lazy $\B$-expansion}, taking the least possible digit at each step. For $x\in[0,x_{\B})$, if the first $N$ digits of the lazy $\B$-expansion of $x$ are given by $a_0,\ldots,a_{N-1}$, then the next digit $a_N$ is the least element in $[\![0,\ceil{\beta_N}-1]\!]$ such that 
\[
	\sum_{n=0}^N\frac{a_n}{\prod_{k=0}^n\beta_k}
	+\sum_{n=N+1}^{\infty}\frac{\ceil{\beta_n}-1}{\prod_{k=0}^n\beta_k}
	\ge x,
\]
or equivalently,
\[
	\sum_{n=0}^N\frac{a_n}{\prod_{k=0}^n\beta_k}
	+\frac{x_{\B^{(N)}}}{\prod_{k=0}^N\beta_k}
	\ge x.
\]
This algorithm is called the \emph{lazy algorithm}. For all $N\in\N$, we have
\[
	\sum_{n=0}^N \frac{a_n}{\prod_{k=0}^n\beta_k}\le x,
\] 
which implies that the lazy algorithm converges, that is, 
\[
	x=\sum_{n=0}^\infty \frac{a_n}{\prod_{k=0}^n\beta_k}.
\] 

We now define the \emph{lazy $\B$-transformation} by
\begin{align*}
	L_{\B}\colon
	&\bigcup_{i=0}^{p-1} \big(\{i\} \times (0,x_{\B^{(i)}}] \big)
	\to \bigcup_{i=0}^{p-1} \big(\{i\} \times (0,x_{\B^{(i)}}] \big),\\
 	&(i,x) \mapsto 
	\begin{cases}
	\big((i+1)\bmod p, \beta_i x\big) 
	& \text{if } x \in (0,x_{\B^{(i)}}-1]\\
	\big((i+1)\bmod p, \beta_i x-\ceil{\beta_i x-x_{\B^{(i+1)}}}\big)
	& \text{if } x \in (x_{\B^{(i)}}-1,x_{\B^{(i)}}].
	\end{cases}
\end{align*}
The lazy $\B$-expansion of $x\in (0,x_{\B}]$ is obtained by alternating the $p$ maps 
\[
	{\pi_2 \circ L_{\B} \circ \delta_i}_{\big| (0,x_{\B^{(i)}}]}
	\colon 
	(0,x_{\B^{(i)}}]\to (0,x_{\B^{(i+1)}}]
\]
for $i\in\Int$.
The following proposition is the analogue of Proposition~\ref{Pro : LexMax} for the lazy $\B$-transformation, which can be proved in a similar fashion.

\begin{proposition}
\label{Pro : LexMin}
For all $x\in(0,x_{\B}]$ and $n\in \N$, we have
\[
	\pi_2 \circ L_{\B}^n \circ \delta_0 (x)
	=\beta_{n-1}\cdots\beta_0 x-\sum_{i=0}^{n-1} \beta_{n-1}\cdots \beta_{i+1}c_i	
\]
where $(c_0,\ldots,c_{n-1})$ is the lexicographically least $n$-tuple in $\prod_{k=0}^{n-1}\ [\![0,\ceil{\beta_k}-1]\!]$ such that $\frac{\sum_{i=0}^{n-1} \beta_{n-1}\cdots \beta_{i+1}c_i}{\beta_{n-1}\cdots\beta_0} + \sum_{m=n}^{\infty}\frac{\ceil{\beta_m}-1}{\prod_{k=0}^m \beta_k} \ge x$. 
\end{proposition}

Note that for each $i\in\Int$, 
\[
	L_{\B}\big(\{i\} \times(x_{\B^{(i)}}-1,x_{\B^{(i)}}] \big)
\subseteq \{(i+1)\bmod p\} \times(x_{\B^{(i+1)}}-1,x_{\B^{(i+1)}}].
\]
Therefore, the lazy $\B$-transformation can be restricted to the domain $\bigcup_{i=0}^{p-1} \big(\{i\} \times(x_{\B^{(i)}}-1,x_{\B^{(i)}}]$. 
The (restricted) lazy $\B$-transformation generates the digits of the lazy $\B$-expan\-sions of real numbers in the interval $(x_{\B}-1,x_{\B}]$ as follows. For all $x\in(x_{\B}-1,x_{\B}]$ and $n\in \N$, the digit $a_n$ in the lazy $\B$-expansion of $x$ is equal to $\ceil{\beta_n\pi_2\big(L_{\B}^n(0,x)\big)-x_{\B^{(n+1)}}}$. Finally, observe that for all $(i,x)\in\bigcup_{i=0}^{p-1} (\{i\} \times (0,x_{\B^{(i)}}])$, there exists $N \in \N$ such that for all $n\ge N$, $L_{\B}^n(i,x)\in \bigcup_{i=0}^{p-1} \big(\{i\}\times(x_{\B^{(i)}}-1,x_{\B^{(i)}}]$. 

\begin{example}
\label{Ex : LbetaRacine13}
Consider again the length-$2$ alternate base $\B=(\frac{1+\sqrt{13}}{2}, \frac{5+\sqrt{13}}{6})$ from Examples~\ref{Ex : TbetaRacine13} and~\ref{Ex : TbetaRacine13SecondPart}. We have $x_{\B}=\frac{5+7\sqrt{13}}{18}\simeq 1.67$ and $x_{\B^{(1)}}=\frac{2 + \sqrt{13}}{3}\simeq 1.86$. The maps ${\pi_2 \circ L_{\B} \circ \delta_0}_{\big| (0,x_{\B}]}\colon (0,x_{\B}]\to(0,x_{\B^{(1)}}]$ and ${\pi_2 \circ L_{\B} \circ \delta_1}_{\big| (0,x_{\B^{(1)}}]}\colon (0,x_{\B^{(1)}}]\to(0,x_{\B}]$ are depicted in Figure~\ref{Fig : LbetaRacine13}. 
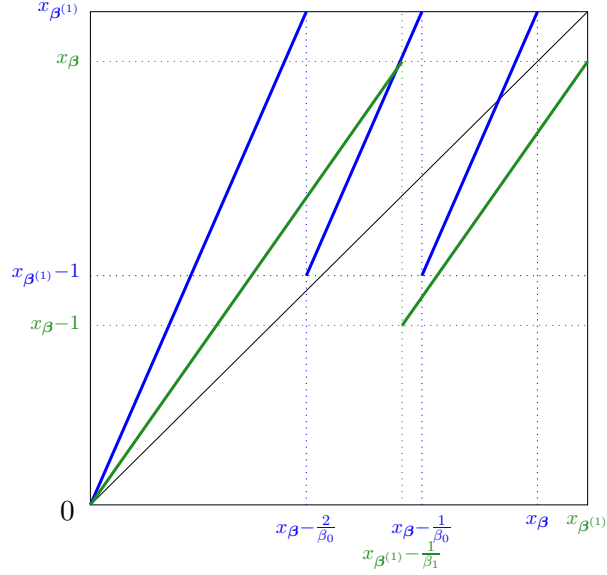
\begin{figure}[htb]
\begin{tikzpicture}[scale=3.5]
\draw[line width=0.0mm ] (0,0) rectangle (1.86852,1.86852); 
\draw[line width=0.0mm ] (0,0)  to (1.86852,1.86852); 
\draw[MyGreen,dotted] (0,1.67994)   to (1.86852,1.67994) ; 
\draw[blue,dotted] (1.67994,0)   to (1.67994,1.86852) ; 
\draw (-0.02,-0.02) node[left]{$0$};
\draw[MyGreen] (1.86852,0) node[below]{$\smath{x_{\B^{(1)}} }$};
\draw[blue] (1.67994,0) node[below]{$\smath{x_{\B}}$};
\draw[MyGreen] (1.17129,-0.1) node[below]{$\smath{x_{\B^{(1)}}{-}\frac{1}{\beta_1}}$};
\draw[blue] (0.81142,0) node[below]{$\smath{x_{\B}{-}\frac{2}{\beta_0}}$};
\draw[blue] (1.24568,0) node[below]{$\smath{x_{\B}{-}\frac{1}{\beta_0}}$};
\draw[blue] (0,1.86852) node[left]{$\smath{x_{\B^{(1)}}}$};
\draw[MyGreen] (0,1.67994) node[left]{$\smath{x_{\B}}$};
\draw[blue] (0,0.86852) node[left]{$\smath{x_{\B^{(1)}}{-}1}$};
\draw[MyGreen] (0,0.67994) node[left]{$\smath{x_{\B}{-}1}$};
\draw[line width=0.4mm,blue] (0,0) to (0.81142,1.86852); 
\draw[line width=0.4mm,blue] (0.81142,0.86852) to (1.24568,1.86852); 
\draw[blue,dotted] (0.81142,0)   to (0.81142,1.86852) ; 
\draw[line width=0.4mm,blue] (1.24568,0.86852) to (1.67994,1.86852); 
\draw[blue,dotted] (1.24568,0)   to (1.24568,1.86852) ;
 \draw[blue,dotted] (0,0.86852)   to (1.86852,0.86852) ;
 \draw[line width=0.4mm,MyGreen] (0,0) to (1.17129,1.67994); 
 \draw[MyGreen,dotted] (1.17129,0)   to (1.17129,1.86852) ; 
  \draw[line width=0.4mm,MyGreen] (1.17129,0.67994) to (1.86852,1.67994); 
  \draw[MyGreen,dotted] (0,0.67994)   to ( 1.86852,0.67994) ;
\end{tikzpicture}
\caption{
The maps ${\pi_2 \circ L_{\B} \circ \delta_0}_{\big| (0,x_{\B}]}$ (blue) and ${\pi_2 \circ L_{\B} \circ \delta_1}_{\big| (0,x_{\B^{(1)}}]}$ (green) with $\B=(\frac{1+\sqrt{13}}{2}, \frac{5+\sqrt{13}}{6})$.}
\label{Fig : LbetaRacine13}
\end{figure}
In Figure~\ref{Fig : First5DigitsLbetaRacine13} we see the computation of the first five digits of the lazy $\B$-expansion of $\frac{1+\sqrt{5}}{5}$. 
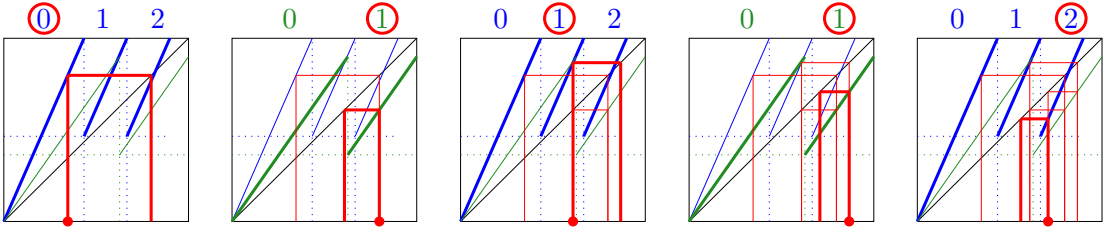
\begin{figure}[htb]
\begin{minipage}{.2\linewidth}
\begin{tikzpicture}[scale=1.3]
\draw[line width=0.0mm ] (0,0) rectangle (1.86852,1.86852); 
\draw[line width=0.0mm ] (0,0)  to (1.86852,1.86852); 
\draw[line width=0.4mm,blue] (0,0) to (0.81142,1.86852); 
\draw[line width=0.4mm,blue] (0.81142,0.86852) to (1.24568,1.86852); 
\draw[blue,dotted] (0.81142,0)   to (0.81142,1.86852) ; 
\draw[line width=0.4mm,blue] (1.24568,0.86852) to (1.67994,1.86852); 
\draw[blue,dotted] (1.24568,0)   to (1.24568,1.86852) ;
 \draw[blue,dotted] (0,0.86852)   to (1.67994,0.86852) ;
 \draw[line width=0.4mm,red] (0.81142/2,2.07) circle [radius=0.15] node {$ $};
 \draw[blue] (0.81142/2,1.86852) node[above]{$0$};
\draw[blue] (1,1.86852) node[above]{$1$};
\draw[blue] (1.5571,1.86852 ) node[above]{$2$};
 \draw[line width=0.05mm,MyGreen] (0,0) to (1.17129,1.67994); 
 \draw[MyGreen,dotted] (1.17129,0)   to (1.17129,1.67994) ; 
  \draw[line width=0.05mm,MyGreen] (1.17129,0.67994) to (1.86852,1.67994); 
  \draw[MyGreen,dotted] (0,0.67994)   to ( 1.86852,0.67994) ;
\filldraw[red] (0.647214,0) circle (1.3pt);
\draw[line width=0.4mm,red] (0.647214,0) to (0.647214,1.49039); 
\draw[line width=0.4mm,red] (0.647214,1.49039) to (1.49039 ,1.49039); 
\draw[line width=0.4mm,red] (1.49039 ,1.49039) to (1.49039 ,0);
\end{tikzpicture}
\end{minipage}%
\begin{minipage}{.2\linewidth}
\begin{tikzpicture}[scale=1.3]
\draw[line width=0.0mm ] (0,0) rectangle (1.86852,1.86852); 
\draw[line width=0.0mm ] (0,0)  to (1.86852,1.86852); 
\draw[line width=0.05mm,blue] (0,0) to (0.81142,1.86852); 
\draw[line width=0.05mm,blue] (0.81142,0.86852) to (1.24568,1.86852); 
\draw[line width=0.05mm,blue] (1.24568,0.86852) to (1.67994,1.86852); 
\draw[blue,dotted] (0.81142,0)   to (0.81142,1.86852) ; 
\draw[blue,dotted] (1.24568,0)   to (1.24568,1.86852) ;
 \draw[blue,dotted] (0,0.86852)   to (1.67994,0.86852) ;
 \draw[line width=0.4mm,MyGreen] (0,0) to (1.17129,1.67994); 
 \draw[MyGreen,dotted] (1.17129,0)   to (1.17129,1.67994) ; 
  \draw[line width=0.4mm,MyGreen] (1.17129,0.67994) to (1.86852,1.67994); 
  \draw[MyGreen,dotted] (0,0.67994)   to ( 1.86852,0.67994) ;
  \draw[MyGreen] (1.17129/2,1.86852) node[above]{$0$};
\draw[MyGreen] (1.51991,1.86852) node[above]{$1$}; 
\draw[line width=0.4mm,red] (1.51991,2.07) circle [radius=0.15] node {$ $};
\draw[line width=0.05mm,red] (0.647214,0) to (0.647214,1.49039); 
\draw[line width=0.05mm,red] (0.647214,1.49039) to (1.49039 ,1.49039); 
\draw[line width=0.05mm,red] (1.49039 ,1.49039) to (1.49039 ,0);
\filldraw[red] (1.49039,0) circle (1.3pt);
\draw[line width=0.4mm,red] (1.49039,0) to (1.49039,1.1376);
\draw[line width=0.4mm,red] (1.49039,1.1376) to (1.1376,1.1376); 
\draw[line width=0.4mm,red] (1.1376,1.1376) to (1.1376 ,0);
\end{tikzpicture}
\end{minipage}%
\begin{minipage}{.2\linewidth}
\begin{tikzpicture}[scale=1.3]
\draw[line width=0.0mm ] (0,0) rectangle (1.86852,1.86852); 
\draw[line width=0.0mm ] (0,0)  to (1.86852,1.86852); 
\draw[line width=0.4mm,blue] (0,0) to (0.81142,1.86852); 
\draw[line width=0.4mm,blue] (0.81142,0.86852) to (1.24568,1.86852); 
\draw[blue,dotted] (0.81142,0)   to (0.81142,1.86852) ; 
\draw[line width=0.4mm,blue] (1.24568,0.86852) to (1.67994,1.86852); 
\draw[blue,dotted] (1.24568,0)   to (1.24568,1.86852) ;
 \draw[blue,dotted] (0,0.86852)   to (1.67994,0.86852) ;
 \draw[blue] (0.81142/2,1.86852) node[above]{$0$};
 \draw[line width=0.4mm,red] (1,2.07) circle [radius=0.15] node {$ $};
\draw[blue] (1,1.86852) node[above]{$1$};
\draw[blue] (1.5571,1.86852 ) node[above]{$2$};
 \draw[line width=0.05mm,MyGreen] (0,0) to (1.17129,1.67994); 
 \draw[MyGreen,dotted] (1.17129,0)   to (1.17129,1.67994) ; 
  \draw[line width=0.05mm,MyGreen] (1.17129,0.67994) to (1.86852,1.67994); 
  \draw[MyGreen,dotted] (0,0.67994)   to ( 1.86852,0.67994) ;
\draw[line width=0.05mm,red] (0.647214,0) to (0.647214,1.49039); 
\draw[line width=0.05mm,red] (0.647214,1.49039) to (1.49039 ,1.49039); 
\draw[line width=0.05mm,red] (1.49039 ,1.49039) to (1.49039 ,0);
\draw[line width=0.05mm,red] (1.49039,0) to (1.49039,1.1376);
\draw[line width=0.05mm,red] (1.49039,1.1376) to (1.1376,1.1376); 
\draw[line width=0.05mm,red] (1.1376,1.1376) to (1.1376 ,0);
\filldraw[red] (1.1376 ,0) circle (1.3pt);
\draw[line width=0.4mm,red] (1.1376,0) to (1.1376,1.61964);
\draw[line width=0.4mm,red] (1.1376,1.61964) to (1.61964,1.61964); 
\draw[line width=0.4mm,red] (1.61964,1.61964)to (1.61964 ,0);
\end{tikzpicture}
\end{minipage}%
\begin{minipage}{.2\linewidth}
\begin{tikzpicture}[scale=1.3]
\draw[line width=0.0mm ] (0,0) rectangle (1.86852,1.86852); 
\draw[line width=0.0mm ] (0,0)  to (1.86852,1.86852); 
\draw[line width=0.05mm,blue] (0,0) to (0.81142,1.86852); 
\draw[line width=0.05mm,blue] (0.81142,0.86852) to (1.24568,1.86852); 
\draw[line width=0.05mm,blue] (1.24568,0.86852) to (1.67994,1.86852); 
\draw[blue,dotted] (0.81142,0)   to (0.81142,1.86852) ; 
\draw[blue,dotted] (1.24568,0)   to (1.24568,1.86852) ;
 \draw[blue,dotted] (0,0.86852)   to (1.67994,0.86852) ;
 \draw[line width=0.4mm,MyGreen] (0,0) to (1.17129,1.67994); 
 \draw[MyGreen,dotted] (1.17129,0)   to (1.17129,1.67994) ; 
  \draw[line width=0.4mm,MyGreen] (1.17129,0.67994) to (1.86852,1.67994); 
  \draw[MyGreen,dotted] (0,0.67994)   to ( 1.86852,0.67994) ;
  \draw[MyGreen] (1.17129/2,1.86852) node[above]{$0$};
   \draw[line width=0.4mm,red] (1.51991,2.07) circle [radius=0.15] node {$ $};
\draw[MyGreen] (1.51991,1.86852) node[above]{$1$}; 
\draw[line width=0.05mm,red] (0.647214,0) to (0.647214,1.49039); 
\draw[line width=0.05mm,red] (0.647214,1.49039) to (1.49039 ,1.49039); 
\draw[line width=0.05mm,red] (1.49039 ,1.49039) to (1.49039 ,0);
\draw[line width=0.05mm,red] (1.49039,0) to (1.49039,1.1376);
\draw[line width=0.05mm,red] (1.49039,1.1376) to (1.1376,1.1376); 
\draw[line width=0.05mm,red] (1.1376,1.1376) to (1.1376 ,0);
\draw[line width=0.05mm,red] (1.1376,0) to (1.1376,1.61964);
\draw[line width=0.05mm,red] (1.1376,1.61964) to (1.61964,1.61964); 
\draw[line width=0.05mm,red] (1.61964,1.61964)to (1.61964 ,0);
\filldraw[red] (1.61964 ,0) circle (1.3pt);
\draw[line width=0.4mm,red] (1.61964,0) to (1.61964,1.32298);
\draw[line width=0.4mm,red] (1.61964,1.32298) to (1.32298,1.32298); 
\draw[line width=0.4mm,red] (1.32298,1.32298) to (1.32298,0);
\end{tikzpicture}
\end{minipage}%
\begin{minipage}{.2\linewidth}
\begin{tikzpicture}[scale=1.3]
\draw[line width=0.0mm ] (0,0) rectangle (1.86852,1.86852); 
\draw[line width=0.0mm ] (0,0)  to (1.86852,1.86852); 
\draw[line width=0.4mm,blue] (0,0) to (0.81142,1.86852); 
\draw[line width=0.4mm,blue] (0.81142,0.86852) to (1.24568,1.86852); 
\draw[blue,dotted] (0.81142,0)   to (0.81142,1.86852) ; 
\draw[line width=0.4mm,blue] (1.24568,0.86852) to (1.67994,1.86852); 
\draw[blue,dotted] (1.24568,0)   to (1.24568,1.86852) ;
 \draw[blue,dotted] (0,0.86852)   to (1.67994,0.86852) ;
 \draw[blue] (0.81142/2,1.86852) node[above]{$0$};
\draw[blue] (1,1.86852) node[above]{$1$};
   \draw[line width=0.4mm,red] (1.5571,2.07) circle [radius=0.15] node {$ $};
\draw[blue] (1.5571,1.86852 ) node[above]{$2$};
 \draw[line width=0.05mm,MyGreen] (0,0) to (1.17129,1.67994); 
 \draw[MyGreen,dotted] (1.17129,0)   to (1.17129,1.67994) ; 
  \draw[line width=0.05mm,MyGreen] (1.17129,0.67994) to (1.86852,1.67994); 
  \draw[MyGreen,dotted] (0,0.67994)   to ( 1.86852,0.67994) ;
\draw[line width=0.05mm,red] (0.647214,0) to (0.647214,1.49039); 
\draw[line width=0.05mm,red] (0.647214,1.49039) to (1.49039 ,1.49039); 
\draw[line width=0.05mm,red] (1.49039 ,1.49039) to (1.49039 ,0);
\draw[line width=0.05mm,red] (1.49039,0) to (1.49039,1.1376);
\draw[line width=0.05mm,red] (1.49039,1.1376) to (1.1376,1.1376); 
\draw[line width=0.05mm,red] (1.1376,1.1376) to (1.1376 ,0);
\draw[line width=0.05mm,red] (1.1376,0) to (1.1376,1.61964);
\draw[line width=0.05mm,red] (1.1376,1.61964) to (1.61964,1.61964); 
\draw[line width=0.05mm,red] (1.61964,1.61964)to (1.61964 ,0);
\draw[line width=0.05mm,red] (1.61964,0) to (1.61964,1.32298);
\draw[line width=0.05mm,red] (1.61964,1.32298) to (1.32298,1.32298); 
\draw[line width=0.05mm,red] (1.32298,1.32298) to (1.32298,0);
\filldraw[red] (1.32298,0) circle (1.3pt);
\draw[line width=0.4mm,red] (1.32298,0) to (1.32298,1.04653);
\draw[line width=0.4mm,red] (1.32298,1.04653) to (1.04653,1.04653); 
\draw[line width=0.4mm,red] (1.04653,1.04653) to (1.04653,0);
\end{tikzpicture}
\end{minipage}%
\caption{The first five digits of the lazy $\B$-expansion of $\frac{1+\sqrt{5}}{5}$ are $01112$ for $\B=(\frac{1+\sqrt{13}}{2}, \frac{5+\sqrt{13}}{6})$.}
\label{Fig : First5DigitsLbetaRacine13}
\end{figure}
\end{example}

\subsection{A note on Cantor bases}
\label{Sec : Cantor}

The greedy algorithm described in Sections~\ref{Sec : Greedy} and~\ref{Sec : Lazy} is well defined in the extended context of Cantor bases, i.e., sequences of real numbers $\B=(\beta_n)_{n\in\N}$ greater than $1$ such that the product $\prod_{n=0}^{\infty}\beta_n$ is infinite~\cite{Charlier&Cisternino:2020}. In this case, the greedy algorithm converge on $[0,1)$: for all $x\in[0,1)$, the computed digits $a_n$ are such that $\sum_{n=0}^{\infty} \frac{a_n}{\prod_{k=0}^n \beta_k}=x$. Therefore, the value $x_{\B}$ defined as in~\eqref{Eq : x_B} is greater than or equal to $1$. However, it might be that $x_{\B}=\infty$. For example, it is the case for the Cantor base given by $\beta_n=1+\frac{1}{n+1}$ for all $n\in\N$.

Note that the restriction of the transformation $\pi_2 \circ T_{\B}^n \circ \delta_0$ to the unit interval $[0,1)$ coincide with the composition $T_{\beta_{n-1}}\circ \cdots \circ T_{\beta_0}$. Thus, when restricted to $[0,1)$, Proposition~\ref{Pro : LexMax} can be reformulated as follows.

\begin{proposition}
\label{Pro : CompositionGreedyCantor}
For all $x\in[0,1)$ and $n\in \N$, we have
\[
	T_{\beta_{n-1}}\circ \cdots \circ T_{\beta_0}(x)
	=\beta_{n-1}\cdots\beta_0 x-\sum_{k=0}^{n-1} \beta_{n-1}\cdots \beta_{k+1}c_k
\]
where $(c_0,\ldots,c_{n-1})$ is the lexicographically greatest $n$-tuple in $\prod_{k=0}^{n-1}\ [\![0,\ceil{\beta_k}-1]\!]$ such that $\frac{\sum_{k=0}^{n-1} \beta_{n-1}\cdots \beta_{k+1}c_k}{\beta_{n-1}\cdots\beta_0} \le x$. 
\end{proposition}

For all $k\in [\![0,n-1]\!]$, the transformation $L_{\beta_k}$ is defined on $(0,x_{\beta_k}]$ and can be restricted to $(x_{\beta_k}-1,x_{\beta_k}]$. So, the restricted transformations $L_{\beta_0},\ldots,L_{\beta_{n-1}}$ cannot be composed to one another in general. Therefore, even if the lazy algorithm can be defined for Cantor bases, provided that $x_{\B}<\infty$, we cannot state an analogue of Proposition~\ref{Pro : CompositionGreedyCantor} in terms of the lazy transformations for Cantor bases.

Even though this paper is mostly concerned with alternate bases, let us emphasize that some results are indeed valid for any sequence $(\beta_n)_{n\in\N}\in (\R_{>1})^{\N}$, and hence for any Cantor base. This is the case of Proposition~\ref{Pro : CompositionGreedyCantor}, Theorem~\ref{Thm : mu}, Corollary~\ref{Cor : lambda} and Proposition~\ref{Pro : Density-Mu-i}.

\section{Dynamical properties of $T_{\B}$}
\label{Section : InvariantMeasure}

In this section, we study the dynamics of the greedy $\B$-transformation. First, we generalize Theorem~\ref{Thm : uniqueMeasure} to the transformation $T_{\B}$ on $[\![0,p-1]\!]\times [0,1)$. Second, we extend the obtained result to the extended transformation $T_{\B}$. Third, we provide a formula for the densities of the measures found in the first two parts. Finally, we compute the frequencies of the digits in the greedy $\B$-expansions.

\subsection{Unique absolutely continuous $T_{\B}$-invariant measure}

In order to generalize Theorem~\ref{Thm : uniqueMeasure} to alternate bases, we start by recalling a result of Lasota and Yorke.

\begin{theorem}
\cite[Theorem 4]{Lasota&Yorke:1982}
\label{Thm : Lasota&Yorke}
Let $T\colon [0,1)\to [0,1)$ be a transformation for which there exists a partition $[a_0,a_1),\ldots,[a_{K-1},a_K)$ of the interval $[0,1)$ with $a_0<\cdots<a_K$ such that for each $k \in [\![0,K-1]\!]$, $T_{\big|[a_k,a_{k+1})}$ is convex, $T(a_k)=0$, $T'(a_k)>0$ and $T'(0)>1$. Then there exists a unique $T$-invariant absolutely continuous probability measure. Furthermore, its density is bounded and decreasing, and the corresponding dynamical system is exact. 
\end{theorem}

We then prove a stability lemma. 

\begin{lemma}
\label{Lem : CompositionPiecewiseLinear}
Let $\mathcal{I}$ be the family of transformations $T\colon [0,1)\to [0,1)$ for which there exist a partition $[a_0,a_1),\ldots,[a_{K-1},a_K)$ of the interval $[0,1)$ with $a_0<\cdots<a_K$ and a slope $s>1$ such that for all $k\in[\![0,K-1]\!]$, $a_{k+1}-a_k\le \frac{1}{s}$ and for all $x\in[a_k,a_{k+1})$, $T(x)=s(x-a_k)$. Then $\mathcal{I}$ is closed under composition.
\end{lemma}

\begin{proof}
Let $S,T\in\mathcal{I}$. Let $[a_0,a_1),\ldots,[a_{K-1},a_K)$ and $[b_0,b_1),\ldots,[b_{L-1},b_L)$ be partitions of the interval $[0,1)$ with $a_0<\cdots<a_K$, $b_0<\cdots<b_L$, and let $s,t>1$ such that for all $k\in[\![0,K-1]\!]$, $a_{k+1}-a_k\le \frac{1}{s}$, for all $\ell\in[\![0,L-1]\!]$, $b_{\ell+1}-b_\ell\le \frac{1}{t}$ 
and for all $x\in[0,1)$, $S(x)=s(x-a_k)$ if $x\in[a_k,a_{k+1})$ and $T(x)=t(x-b_\ell)$ if $x\in[b_\ell,b_{\ell+1})$. 
For each $k\in[\![0,K-1]\!]$, define $L_k$ to be the greatest $\ell\in[\![0,L-1]\!]$ such that $a_k+\frac{b_\ell}{s}< a_{k+1}$. Consider the partition
\begin{align*}
&\Big[a_0+\frac{b_0}{s},a_0+\frac{b_1}{s}\Big),\ldots,
\Big[a_0+\frac{b_{L_0-1}}{s},a_0+\frac{b_{L_0}}{s}\Big),
\Big[a_0+\frac{b_{L_0}}{s},a_1\Big)\\
&\vdots \\
&\Big[a_{K-1}+\frac{b_0}{s},a_{K-1}+\frac{b_1}{s}\Big),\ldots,
\Big[a_{K-1}+\frac{b_{L_{K-1}-1}}{s},a_{K-1}+\frac{b_{L_{K-1}}}{s}\Big),
\Big[a_{K-1}+\frac{b_{L_{K-1}}}{s},a_K\Big)
\end{align*} 
of the interval $[0,1)$. For each $k\in[\![0,K-1]\!]$ and $\ell\in[\![0,L_k-1]\!]$, $a_k+\frac{b_{\ell+1}}{s}-a_k-\frac{b_\ell}{s}\le \frac{1}{ts}$ and $a_{k+1}-a_k-\frac{b_{L_k}}{s}=(a_{k+1}-a_k-\frac{b_{L_k+1}}{s})+\frac{b_{L_k+1}-b_{L_k}}{s}\le \frac{1}{ts}$. Now, let $x\in[0,1)$ and $k\in [\![0,K-1]\!]$ be such that $x\in[a_k,a_{k+1})$. Then $S(x)=s(x-a_k)\in[0,1)$. We distinguish two cases: either there exists $\ell\in [\![0,L_k-1]\!]$ such that $x\in[a_k+\frac{b_\ell}{s},a_k+\frac{b_{\ell+1}}{s})$, or $x\in[a_k+\frac{b_{L_k}}{s},a_{k+1})$. In the former case, $S(x)\in[b_\ell,b_{\ell+1})$ and $T\circ S(x)=t(S(x)-b_\ell)=ts(x-(a_k+\frac{b_\ell}{s}))$. In the latter case, since $a_{k+1}-a_k\le \frac{b_{L_k+1}}{s}$, we get that $S(x)\in[b_{L_k},b_{L_k+1})$ and hence that $T\circ S(x)=t(S(x)-b_{L_k})=ts(x-(a_k+\frac{b_{L_k}}{s}))$. This shows that the composition $T\circ S$ belongs to $\mathcal{I}$.
\end{proof}

The following theorem provides us with the main tool for the construction of a $T_{\B}$-invariant measure.

\begin{theorem}
\label{Thm : mu}
For all $n\in\N_{\ge 1}$ and all $\beta_0,\ldots,\beta_{n-1}>1$, there exists a unique $(T_{\beta_{n-1}}\circ \cdots \circ T_{\beta_0})$-invariant absolutely continuous probability measure $\mu$ on $\mathcal{B}([0,1))$.  Furthermore, the measure $\mu$ is equivalent to the Lebesgue measure on $\mathcal{B}([0,1))$, its density is bounded and decreasing, and the dynamical system $([0,1),\mathcal{B}([0,1)),\mu,T_{\beta_{n-1}}\circ \cdots \circ T_{\beta_0})$ is exact and has entropy $\log(\beta_{n-1}\cdots\beta_0)$.
\end{theorem}

\begin{proof}
The existence of a unique  $(T_{\beta_{n-1}}\circ \cdots \circ T_{\beta_0})$-invariant absolutely continuous probability measure $\mu$ on $\mathcal{B}([0,1))$, the fact that its density is bounded and decreasing, and the exactness of the corresponding dynamical system follow from Theorem~\ref{Thm : Lasota&Yorke} and Lemma~\ref{Lem : CompositionPiecewiseLinear}. With a similar argument as in~\cite{Dajani&Kalle:2010}, we can conclude that $\frac{d\mu}{d\lambda}>0$ $\lambda$-a.e.\ on $[0,1)$. It follows that $\mu$ is equivalent to the Lebesgue measure on $\mathcal{B}([0,1))$. Moreover, the entropy equals $\log(\beta_{n-1}\cdots\beta_0)$ since $T_{\beta_{n-1}}\circ \cdots \circ T_{\beta_0}$ is a piecewise linear transformation of constant slope $\beta_{n-1}\cdots\beta_0$ \cite{Dajani&Kalle:2021,Rohlin:1961}.
\end{proof}

The following consequence of Theorem~\ref{Thm : mu} will be useful for proving our generalization of Theorem~\ref{Thm : uniqueMeasure}.

\begin{corollary}
\label{Cor : lambda}
Let $n\in\N_{\ge 1}$ and $\beta_0,\ldots,\beta_{n-1}>1$. Then for all $B\in\mathcal{B}([0,1))$ such that $(T_{\beta_{n-1}}\circ \cdots \circ T_{\beta_0})^{-1}(B)=B$, we have $\lambda(B)\in\{0,1\}$.
\end{corollary}

For each $i\in\Int$, we let $\mu_{\B,i}$ denote the unique $(T_{\beta_{i-1}}\circ\cdots \circ T_{\beta_{i-p}})$-invariant absolutely continuous probability measure given by Theorem~\ref{Thm : mu}. We use the convention that for all $n\in \Z$, $\mu_{\B,n}=\mu_{\B,n \bmod p}$.  Let us define a probability measure $\mu_{\B}$ on the $\sigma$-algebra 
\[
	\mathcal{T}_p
	=\Bigg\{
	\bigcup_{i=0}^{p-1} (\{i\}\times B_i)				\colon \forall i\in\Int,\ B_i\in\mathcal{B}([0,1))
	\Bigg\}
\] 
over $\Int\times [0,1)$ as follows. 
For all $B_0,\ldots,B_{p-1}\in\mathcal{B}([0,1))$, we set
\[
	\mu_{\B}\Bigg(\bigcup_{i=0}^{p-1}(\{i\}\times B_i)\Bigg)
	=\frac{1}{p}\sum_{i=0}^{p-1}\mu_{\B,i}(B_i).
\]

We now study the properties of the probability measure $\mu_{\B}$.

\begin{lemma}
\label{Lem : Julien}
For $i\in\Int$, we have $\mu_{\B,i}=\mu_{\B,i-1}\circ T_{\beta_{i-1}}^{-1}$.
\end{lemma}

\begin{proof}
Let $i\in\Int$. By definition of $\mu_{\B,i}$, it suffices to show that $\mu_{\B,i-1}\circ T_{\beta_{i-1}}^{-1}$ is a $(T_{\beta_{i-1}}\circ \cdots \circ T_{\beta_{i-p}})$-invariant absolutely continuous probability measure on $\mathcal{B}([0,1))$. First, we have $\mu_{\B,i-1}\big(T_{\beta_{i-1}}^{-1}([0,1))\big)=\mu_{\B,i-1}([0,1))=1$. Second, for all $B\in\mathcal{B}([0,1))$, we have
\begin{align*}
	\mu_{\B,i-1}\circ T_{\beta_{i-1}}^{-1}\big((T_{\beta_{i-1}}\circ \cdots \circ T_{\beta_{i-p}})^{-1}(B)\big)
	&=\mu_{\B,i-1}\big((T_{\beta_{i-1}}\circ \cdots \circ T_{\beta_{i-p}}\circ T_{\beta_{i-p-1}})^{-1}(B)\big)\\
	&=\mu_{\B,i-1}\big((T_{\beta_{i-2}}\circ \cdots \circ T_{\beta_{i-p-1}})^{-1}(T_{\beta_{i-1}}^{-1}(B))\big)\\
	&=\mu_{\B,i-1}\big(T_{\beta_{i-1}}^{-1}(B)\big).
\end{align*}
Third, for all $B\in\mathcal{B}([0,1))$ such that $\lambda(B)=0$, we get that $\lambda(T_{\beta_{i-1}}^{-1}(B))=0$ by Remark~\ref{Rem : NonSing}, and hence that $\mu_{\B,i-1}(T_{\beta_{i-1}}^{-1}(B))=0$ since $\mu_{\B,i-1}$ is absolutely continuous. 
\end{proof}

\begin{proposition}
\label{Pro : MuBInvariant}
The measure $\mu_{\B}$ is $T_{\B}$-invariant.
\end{proposition}

\begin{proof}
For all $B_0,\ldots,B_{p-1}\in \mathcal{B}([0,1))$, 
\begin{align*}
	\mu_{\B}\Bigg(T_{\B}^{-1}
	\Bigg( \bigcup_{i=0}^{p-1} (\{i\}\times B_i)\Bigg)
	\Bigg)
	&=\mu_{\B}\Bigg(
	\bigcup_{i=0}^{p-1}T_{\B}^{-1}(\{i\}\times B_i)
	\Bigg)\\
	&=\mu_{\B}\Bigg(
	\bigcup_{i=0}^{p-1}\big(\{(i-1)\bmod p\}\times T_{\beta_{i-1}}^{-1}(B_i)\big)
	\Bigg)\\	
	&=\frac{1}{p}\sum_{i=0}^{p-1}\mu_{\B,i-1}(T_{\beta_{i-1}}^{-1}(B_i))\\
	&=\frac{1}{p}\sum_{i=0}^{p-1}\mu_{\B,i}(B_i)\\
	&=	\mu_{\B}\Bigg(
	\bigcup_{i=0}^{p-1} (\{i\}\times B_i)
	\Bigg)
\end{align*}
where we applied Lemma~\ref{Lem : Julien} for the fourth equality.
\end{proof}

\begin{corollary}
\label{Cor : GreedyDynamicalSystem}
The quadruple $\big( \Int \times [0,1),\mathcal{T}_p,\mu_{\B},T_{\B}\big)$ is a dynamical system.
\end{corollary}

Let us define a new measure $\lambda_p$ over the $\sigma$-algebra $\mathcal{T}_p$. For all $B_0,\ldots,B_{p-1}\in\mathcal{B}([0,1))$, we set
\[
	\lambda_p\Bigg(\bigcup_{i=0}^{p-1} (\{i\}\times B_i)\Bigg)
	=\frac{1}{p}\sum_{i=0}^{p-1}\lambda(B_i).
\]
We call this measure the \emph{$p$-Lebesgue measure} on $\mathcal{T}_p$. 

\begin{proposition}
\label{Pro : Equiv}
The measure $\mu_{\B}$ is equivalent to the $p$-Lebesgue measure on $\mathcal{T}_p$. 
\end{proposition}

\begin{proof}
This follows from the fact that the $p$ measures $\mu_{\B,0},\ldots,\mu_{\B,p-1}$ are equivalent to the Lebesgue measure $\lambda$ on $\mathcal{B}([0,1))$.
\end{proof}

Next, we compute the entropy of the dynamical  system $\big( \Int \times [0,1),\mathcal{T}_p, \mu_{\B}, T_{\B}\big)$. To do so, we consider the $p$ induced transformations 
\[
	T_{\B,i}\colon {\{i\}\times[0,1)}\to {\{i\}\times[0,1)},\
	(i,x)\mapsto T_{\B}^p(i,x)
\]
for $i\in\Int$. Note that indeed, for all $(i,x)\in\Int\times[0,1)$, the first return of $(i,x)$ to ${\{i\}\times[0,1)}$ is equal to $p$. Thus $T_{\B,i}={T_{\B}^p}_{\big| \{i\} \times [0,1)}$. As is well know~\cite{Dajani&Kalle:2021}, for each $i\in\Int$, the induced transformation $T_{\B,i}$ is measure preserving with respect to the measure $\nu_{\B,i}$ on the $\sigma$-algebra $\{\{i\}\times B\colon B\in\mathcal{B}([0,1))\}$ defined as follows: for all $B \in \mathcal{B}([0,1))$,
\[
	\nu_{\B,i}(\{i\}\times B)=p \mu_{\B}(\{i\}\times B).
\] 

\begin{lemma}
\label{Lem : Entropy}
For every $i\in\Int$, the map ${\delta_i}_{\big|[0,1)}\colon [0,1)\to \{i\}\times[0,1),\ x\mapsto (i,x)$ defines an isomorphism between the dynamical systems 
\[
	\big(
	[0,1),
	\mathcal{B}([0,1)),
	\mu_{\B,i},
	T_{\beta_{i-1}}\circ\cdots\circ T_{\beta_{i-p}}
	\big)
\] 
and 
\[
	\big(
	\{i\} \times [0,1),
	\{\{i\}\times B\colon B\in\mathcal{B}([0,1))\},
	\nu_{\B,i},
	T_{\B,i}
	\big)
\]
\end{lemma}

\begin{proof}
This is a straightforward verification.
\end{proof}

\begin{proposition}
\label{Pro : Entropy}
The entropy $h_{\B}$ of the dynamical system $\big( \Int \times [0,1),\mathcal{T}_p, \mu_{\B}, T_{\B}\big)$ is $\frac{1}{p} \log(\beta_{p-1}\cdots \beta_0)$.
\end{proposition}

\begin{proof}
Let $i\in\Int$. By Abramov's formula~\cite{Abramov:1959}, we have
\[
	h_{\B}
	=\mu_{\B}(\{i\} \times [0,1)) \, h_{\B,i}
	=\frac1p \, h_{\B,i}.
\]
where $h_{\B,i}$ denotes the entropy of the induced dynamical system $\big(\{i\} \times [0,1),\{\{i\} \times B\colon B\in \mathcal{B}([0,1))\}, \nu_{\B,i},T_{\B,i})$. Since the entropy is an isomorphism invariant, it follows from Theorem~\ref{Thm : mu} and Lemma~\ref{Lem : Entropy} that $h_{\B,i}= \log(\beta_{p-1}\cdots \beta_0)$. Hence the conclusion.
\end{proof}

Finally, we prove that any $T_{\B}$-invariant set has $p$-Lebesgue measure $0$ or $1$.

\begin{proposition}
\label{Pro : pseudo-Ergodic}
For all $A\in\mathcal{T}_p$ such that $T_{\B}^{-1}(A)=A$, we have $\lambda_p(A)\in\{0,1\}$.
\end{proposition}

\begin{proof}
Let $B_0,\ldots,B_{p-1}$ be sets in $\mathcal{B}([0,1))$ such that 
\[
	T_{\B}^{-1}\Bigg(\bigcup_{i=0}^{p-1}(\{i\}\times B_i)\Bigg)
	=\bigcup_{i=0}^{p-1}(\{i\}\times B_i).
\]
This implies that 
\begin{equation}
\label{Eq : propErgodic}
	T_{\beta_{i-1}}^{-1}(B_i)=B_{(i-1)\bmod p}
	\quad\text{for all }i\in\Int.
\end{equation}
We use the convention that $B_n=B_{n \bmod p}$ for all $n \in \Z$. An easy induction yields that for all $i \in \Int$ and $n\in \N$, $(T_{\beta_{i-1}}\circ \cdots\circ T_{\beta_{i-n}})^{-1}(B_i)=B_{i-n}$. In particular, for $n=p$, we get that for each $i \in \Int$, $(T_{\beta_{i-1}}\circ \cdots\circ T_{\beta_{i-p}})^{-1}(B_i)=B_i$. By Corollary~\ref{Cor : lambda}, for each $i\in\Int$, $\lambda(B_{i})\in\{0,1\}$. By definition of $\lambda_p$, in order to conclude, it suffices to show that either $\lambda(B_i)=0$ for all $i\in\Int$, or $\lambda(B_i)=1$ for all $i\in\Int$. From~\eqref{Eq : propErgodic} and Remark~\ref{Rem : NonSing}, we get that for each $i\in\Int$, $\lambda(B_i)=0$ if and only if $\lambda(B_{i+1})=0$. The conclusion follows.
\end{proof}

We are now able to state the announced generalization of Theorem~\ref{Thm : uniqueMeasure} to alternate bases.

\begin{theorem}
\label{Thm : GlobalTheoremAlternateBase}
The measure $\mu_{\B}$ is the unique $T_{\B}$-invariant probability measure on $\mathcal{T}_p$ that is absolutely continuous with respect to $\lambda_p$. Furthermore, $\mu_{\B}$ is equivalent to $\lambda_p$ on $\mathcal{T}_p$ and the dynamical system $( \Int \times [0,1), \mathcal{T}_p, \mu_{\B}, T_{\B})$ is ergodic and has entropy $\frac{1}{p}\log(\beta_{p-1}\cdots\beta_0)$.
\end{theorem}

\begin{proof}
By Propositions~\ref{Pro : MuBInvariant} and~\ref{Pro : Equiv}, $\mu_{\B}$ is a $T_{\B}$-invariant probability measure that is absolutely continuous with respect to $\lambda_p$ on $\mathcal{B}([0,1))$. Then we get from Proposition~\ref{Pro : pseudo-Ergodic} that for all $A\in\mathcal{T}_p$ such that $T_{\B}^{-1}(A)=A$, we have $\mu_{\B}(A)\in\{0,1\}$. Therefore, the dynamical system $( \Int \times [0,1), \mathcal{T}_p, \mu_{\B}, T_{\B})$ is ergodic. Now, we obtain that the measure $\mu_{\B}$ is unique as a well-known consequence of the Ergodic Theorem, see~\cite[Theorem 3.1.2]{Dajani&Kalle:2021}. The equivalence between $\mu_{\B}$ and $\lambda_p$ and the entropy of the system were already obtained in Propositions~\ref{Pro : Equiv} and~\ref{Pro : Entropy}.
\end{proof}

For $p$ greater than $1$, the dynamical system $(\Int \times [0,1),\mathcal{T}_p,\mu_{\B}, T_{\B})$ is not exact even though for all $i\in \Int$, the dynamical systems $([0,1), \mathcal{B}([0,1)), \mu_{\B,i}, T_{\beta_{i-1}}\circ \cdots\circ T_{\beta_{i-p}})$ are exact. It suffices to note that the dynamical system $(\Int \times [0,1),\mathcal{T}_p,\mu_{\B}, T_{\B}^p)$ is not ergodic for $p>1$. Indeed, $T_{\B}^{-p}(\{0\}\times [0,1))=\{0\}\times [0,1)$ whereas $\mu_{\B}(\{0\}\times [0,1))=\frac{1}{p}$.

\subsection{Extended measure}

In order to study the dynamics of the extended greedy $\B$-transformation, we extend the definitions of the measures $\mu_{\B}$ and $\lambda_p$. First, we define a new $\sigma$-algebra $\mathcal{T}_{\B}$ on $\bigcup_{i=0}^{p-1} \big(\{i\}\times[0,x_{\B^{(i)}}) \big)$ as follows:
\[
	\mathcal{T}_{\B}
	=\Bigg\{
	\bigcup_{i=0}^{p-1} 
	(\{i\}\times B_i)\colon 
	\forall i\in\Int,\ 
	B_i\in\mathcal{B}([0, x_{\B^{(i)}}))
	\Bigg\}.
\]
Second, we extend the domain of the measures $\mu_{\B}$ and $\lambda_p$ to $\mathcal{T}_{\B}$ (while keeping the same notation) as follows. 
For $A\in\mathcal{T}_{\B}$, we set
$\mu_{\B}(A)=\mu_{\B}\big(A\cap\big(\Int\times [0,1)\big)\big)$ and $\lambda_p(A)= \lambda_p\big(A\cap\big(\Int\times [0,1)\big)\big)$.

\begin{theorem}
\label{Thm : GlobalTheoremAlternateBaseExtended}
The measure $\mu_{\B}$ is the unique $T_{\B}$-invariant probability measure on $\mathcal{T}_{\B}$ that is absolutely continuous with respect to $\lambda_p$. Furthermore, $\mu_{\B}$ is equivalent to $\lambda_p$ on $\mathcal{T}_{\B}$ and the dynamical system $(\bigcup_{i=0}^{p-1} \big(\{i\}\times[0,x_{\B^{(i)}}) \big), \mathcal{T}_{\B}, \mu_{\B}, T_{\B})$ is ergodic and has entropy $\frac{1}{p}\log(\beta_{p-1}\cdots\beta_0)$.
\end{theorem}

\begin{proof}
Clearly, $\mu_{\B}$ is a probability measure on $\mathcal{T}_{\B}$. For all $A\in\mathcal{T}_{\B}$, we have 
\begin{align*}
	\mu_{\B}(T_{\B}^{-1}(A))
	&=\mu_{\B}\big(T_{\B}^{-1}(A)\cap (\Int\times[0,1))\big)\\
	&=\mu_{\B}\big(T_{\B}^{-1}\big(A\cap (\Int\times[0,1))\big)\cap (\Int\times[0,1))\big)\\
	&=\mu_{\B}\big(T_{\B}^{-1}\big(A\cap (\Int\times[0,1))\big)\big)\\
	&=\mu_{\B}\big(A\cap (\Int\times[0,1))\big)\\
	&=\mu_{\B}(A)
\end{align*}
where we used Proposition~\ref{Pro : MuBInvariant} for the fourth equality.
This shows that $\mu_{\B}$ is $T_{\B}$-invariant on $\mathcal{T}_{\B}$.
The conclusion then follows from the fact that the identity map from $[\![0,p-1]\!]\times [0,1)$ to $\bigcup_{i=0}^{p-1} \big(\{i\}\times[0,x_{\B^{(i)}}) \big)$ defines an isomorphism between the dynamical systems $([\![0,p-1]\!]\times [0,1), \mathcal{F}_{p}, \mu_{\B}, T_{\B})$ and
$(\bigcup_{i=0}^{p-1} \big(\{i\}\times[0,x_{\B^{(i)}}) \big), \mathcal{T}_{\B}, \mu_{\B}, T_{\B})$.
\end{proof}

\subsection{Densities}\label{Section : Densities}

In the next proposition, we express the density of the unique measure given in Theorem~\ref{Thm : mu}.

\begin{proposition}
\label{Pro : Density-Mu-i}
Consider $n\in\N_{\ge 1}$ and $\beta_0,\ldots,\beta_{n-1}>1$. Suppose that
\begin{itemize}
\item $K$ is the number of not onto branches of $T_{\beta_{n-1}}\circ\cdots \circ T_{\beta_{0}}$
\item for $j\in [\![1,K ]\!]$, $c_j$ is the right-hand side endpoint of the domain of the $j$-th not onto branche of $T_{\beta_{n-1}}\circ\cdots \circ T_{\beta_{0}}$
\item $T\colon [0,1)\to [0,1)$ is the transformation defined by $T(x)=T_{\beta_{n-1}}\circ\cdots \circ T_{\beta_{0}}(x)$ for $x\notin\{c_1,\ldots,c_K\}$ and $T(c_j)=\lim_{x \to c_j^{-}}T_{\beta_{n-1}}\circ\cdots \circ T_{\beta_{0}} (x)$ for $j\in [\![1,K]\!]$ 
\item $S$ is the matrix defined by $S=(S_{i,j})_{1\le i,j,\le K}$ where 
\[
	S_{i,j}
	=\sum_{m=1}^\infty 
	\frac{\delta(T^{m}( c_i)>c_j) }{(\beta_{n-1}\cdots\beta_0)^m},
\]
where $\delta(P)$ equals $1$ when the property $P$ is satisfied and $0$ otherwise
\item $1$ is not an eigenvalue of $S$
\item $d_0=1$ and $\begin{pmatrix}
d_1 \cdots d_K
\end{pmatrix}
=\begin{pmatrix}
1 \cdots 1
\end{pmatrix}(-S+ \text{Id}_K)^{-1}$
\item $C=\int_0^1 \Big(d_0 + \sum_{j=1}^{K}d_j\sum_{m=1}^{\infty} \frac{\chi_{[0,T^{m}(c_j)]}}{(\beta_{n-1}\cdots\beta_0)^m} \Big)\, d\lambda$  is the normalization constant.
\end{itemize} 
Then the density of the $(T_{\beta_{n-1}}\circ \cdots \circ T_{\beta_0})$-invariant measure given by Theorem~\ref{Thm : mu} with respect to the Lebesgue measure is  
\begin{equation}
\label{Eq : Densities}
	\frac{1}{C}\Bigg(d_0 
	+ \sum_{j=1}^{K}d_j
	\sum_{m=1}^{\infty} 	
	\frac{\chi_{[0,T^{m}(c_j)]}}{(\beta_{n-1}\cdots\beta_0)^m}
	\Bigg).
\end{equation}
\end{proposition}

\begin{proof}
This is an application of the formula given in~\cite{Gora:2009}. 
\end{proof}

Note that the only hypothesis in the statement of Proposition~\ref{Pro : Density-Mu-i} is that $1$ is not an eigenvalue of the matrix $S$. In \cite{Gora:2009} Gora conjectured that this condition is equivalent to the exactness of the dynamical system, which is a property  we know to be satisfied by Theorem~\ref{Thm : mu}. 

\begin{example}
Consider once again the alternate base $\B=(\frac{1+\sqrt{13}}{2}, \frac{5+\sqrt{13}}{6})$. The composition $T_{\beta_1} \circ T_{\beta_0}$ is depicted in Figure~\ref{Fig : TBsquared13-01}.
\begin{figure}[htb]
\centering
\begin{tikzpicture}
6.71975
\draw[line width=0.0mm ] (0,0) rectangle (4,4); 
\draw[dotted,line width=0.15mm ] (0,1.73703) to (4,1.73703); 
\draw[line width=0.0mm ] (0,0) to (4,4); 
\draw (-0.1,-0.2) node[left]{$0$};
\draw (4,0) node[below]{$1$};
\draw (0,4) node[left]{$1$};
\draw (0,1.73703) node[left]{$\smath{\tfrac{1}{\beta_0}}$};
\draw (1.2111,0) node[below]{$\smath{\tfrac{1}{\beta_1\beta_0}}$};
\draw (1.73703,0) node[below]{$\smath{\tfrac{1}{\beta_0}}$};
\draw ( 2.94814,0) node[below]{$\smath{\tfrac{\beta_1+1}{\beta_1\beta_0}}$};
\draw ( 3.47407,0) node[below]{$\smath{\tfrac{2}{\beta_0}}$};
\draw[line width=0.4mm,blue] (0,0) to (1.2111,4); 
\draw[line width=0.4mm,blue] (1.2111,0) to (1.73703,1.73703); 
\draw[line width=0.4mm,blue] (1.73703,0) to (2.94814,4); 
\draw[line width=0.4mm,blue] (2.94814,0) to (3.47407,1.73703); 
\draw[line width=0.4mm,blue] (3.47407,0) to (4,1.73703); 
\end{tikzpicture}
\caption{The composition $T_{\beta_1}\circ T_{\beta_0}$  with $\B=(\frac{1+\sqrt{13}}{2}, \frac{5+\sqrt{13}}{6})$.}
\label{Fig : TBsquared13-01}
\end{figure}
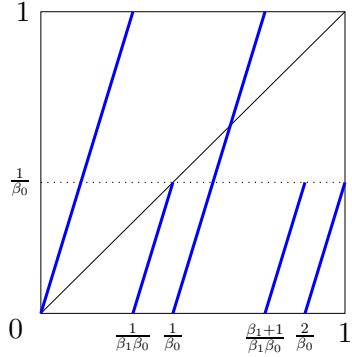 
Since $\frac{1}{\beta_0}=\beta_1-1$, keeping the notation of Proposition~\ref{Pro : Density-Mu-i}, we have $K=3$, $c_1=\frac{1}{\beta_0}$, $c_2=\frac{2}{\beta_0}$ and $c_3=1$. We have $T(c_1)=T(c_2)=T(c_3)=c_1$. Therefore, all elements in $S$ equal $0$, $d_0=d_1=d_2=d_3=1$ and $C=1+ \frac{3}{\beta_0(\beta_1\beta_0-1)}=1+\frac{3}{\beta_0^2}$. The density of the unique absolutely continuous $(T_{\beta_1} \circ T_{\beta_0})$-invariant probability measure is
\[
	\frac1C \Big(1+\frac{3}{\beta_0}\chi_{[0,\frac{1}{\beta_0}]}\Big).
\] 
For example, $\mu\big([0,\frac{1}{\beta_0})\big)
= \frac{13+\sqrt{13}}{26}$. Moreover, it can be checked that $\mu\big((T_{\beta_1} \circ T_{\beta_0})^{-1}[0,\frac{1}{\beta_0})\big)=\mu\big([0,\frac{1}{\beta_0})\big)$.
\end{example}

We obtain a formula for the density $\frac{d\mu_{\B}}{d\lambda_p}$ by using the densities $\frac{d\mu_{\B,i}}{d\lambda}$ for $i\in \Int$ given in Proposition~\ref{Pro : Density-Mu-i}. We first need a lemma.

\begin{lemma}
\label{Lem : Samuel}
For all $i\in\Int$, all sets $B\in\mathcal{B}([0,1))$ and all $\mathcal{B}([0,1))$-measurable functions $f\colon[0,1)\to [0,\infty)$, the map $f\circ\pi_2\colon\Int\times[0,1)\to[0,\infty)$ is $\mathcal{T}_p$-measurable and 
\[
	\int_{\{i\}\times B} f\circ \pi_2\ d\lambda_p
	=\frac{1}{p}\int_{B} f d\lambda.
\]
\end{lemma}

\begin{proof}
This follows from the definition of the Lebesgue integral via simple functions.
\end{proof}

\begin{proposition}
\label{Pro : DensityMuB}
The density $\frac{d\mu_{\B}}{d\lambda_p}$ of $\mu_{\B}$ with respect to the $p$-Lebesgue measure on $\mathcal{T}_p$ is
\begin{equation}
\label{Eq : densite}
	\sum_{i=0}^{p-1}
\Bigg(\frac{d\mu_{\B,i}}{d\lambda}\circ \pi_2\Bigg)\cdot \chi_{\{i\} \times [0,1)}.
\end{equation}
\end{proposition}

\begin{proof}
Let $A\in\mathcal{T}_p$ and let $B_0,\ldots,B_{p-1}\in\mathcal{B}([0,1))$ such that $A=\bigcup_{i=0}^{p-1}(\{i\}\times B_i)$. It follows from  Lemma~\ref{Lem : Samuel} that 
\begin{align*}
	\int\limits_{A} \sum_{i=0}^{p-1} \Bigg(\frac{d\mu_{\B,i}}{d\lambda}\circ \pi_2\Bigg)\cdot \chi_{\{i\} \times [0,1)}\ d\lambda_p
	=&\sum_{i=0}^{p-1}
	\int\limits_{\{i\}\times B_i} \frac{d\mu_{\B,i}}{d\lambda}\circ \pi_2\ d\lambda_p\\
	=&\frac{1}{p}\sum_{i=0}^{p-1}
	\int\limits_{B_i} \frac{d\mu_{\B,i}}{d\lambda}\ d\lambda\\
	=&\frac{1}{p}\sum_{i=0}^{p-1}\mu_{\B,i}(B_i)\\
	=&	\mu_{\B}(A).
\end{align*}
\end{proof}

Note that the formula~\eqref{Eq : densite} also holds for the extended measures $\mu_{\B}$ and $\lambda_p$ on $\mathcal{T}_{\B}$.

\subsection{Frequencies}
\label{Section : Frequencies}

We now turn to the frequencies of the digits in the greedy $\B$-expansions of real numbers in the interval $[0,1)$. Recall that the frequency of a digit $d$ occurring in the greedy $\B$-expansion $a_0a_1a_2\cdots$ of a real number $x$ in $[0,1)$ is equal to
\[
	\lim_{n\to \infty} \frac1n 
	\#\{0\le k<n\colon a_k=d\},
\]
provided that this limit converges. 

\begin{proposition}
For $\lambda$-almost all $x\in[0,1)$, the frequency of any digit $d$ occurring in the greedy $\B$-expansion of $x$ exists and is equal to
\[
	\frac{1}{p}\sum_{i=0}^{p-1} 
	\mu_{\B,i} \Big(\big[\tfrac{d}{\beta_i},\tfrac{d+1}{\beta_i} \big)\cap [0,1) \Big).
\]
\end{proposition}

\begin{proof} 
Let $x\in[0,1)$ and let $d$ be a digit occurring in $d_{\B}(x)=a_0a_1a_2\cdots$. Then for all $k\in\N$, $a_k=d$ if and only if $\pi_2(T_{\B}^k(0,x))\in [\frac{d}{\beta_k},\frac{d+1}{\beta_k})\cap [0,1)$. Moreover, since for all $k\in\N$, $T_{\B}^k(0,x)\in\{k\bmod p\}\times [0,1)$, we have 
\begin{align*}
	\chi_{[\frac{d}{\beta_k},\frac{d+1}{\beta_k})\cap [0,1)} 
	\big(\pi_2\big(T_{\B}^k(0,x)\big)\big)
	&=\chi_{\{k\bmod p\}\times
	\big([\frac{d}{\beta_k},\frac{d+1}{\beta_k})\cap [0,1)\big)} 
	\big(T_{\B}^k(0,x)\big)\\
	&=\sum_{i=0}^{p-1}
	\chi_{\{i\}\times
	\big([\frac{d}{\beta_i},\frac{d+1}{\beta_i})\cap [0,1)\big)} 
	\big(T_{\B}^k(0,x)\big).
\end{align*}
Therefore, if it exists, the frequency of $d$ in $d_{\B}(x)$ is equal to
\[
	\lim_{n\to\infty}
	\frac1n \sum_{k=0}^{n-1} 
	\sum_{i=0}^{p-1}
	\chi_{\{i\}\times
	\big([\frac{d}{\beta_i},\frac{d+1}{\beta_i})\cap [0,1)\big)} 
	\big(T_{\B}^k(0,x)\big).
\]
Yet, for each $i\in\Int$ and for $\mu_{\B}$-almost all $y\in\Int\times[0,1)$, we have
\begin{align*}
	\lim_{n\to \infty}  
		\frac1n
		\sum_{k=0}^{n-1} 
	\chi_{\{i\}\times
	\big([\frac{d}{\beta_i},\frac{d+1}{\beta_i})\cap [0,1)\big)} 
	\big(T_{\B}^k(y)\big)
	&=	\int_{\Int\times[0,1)} 
	\chi_{\{i\}\times
	\big([\frac{d}{\beta_i},\frac{d+1}{\beta_i})\cap [0,1)\big)} 
	d\mu_{\B}\\
	&=\mu_{\B} \Big(\{i\}\times  \big(\big[\tfrac{d}{\beta_i},\tfrac{d+1}{\beta_i} \big)\cap [0,1)\big) \Big)\\
	&=\frac{1}{p}\mu_{\B,i} \Big(\big[\tfrac{d}{\beta_i},\tfrac{d+1}{\beta_i} \big)\cap [0,1)\Big)
\end{align*}
where we used Theorem~\ref{Thm : GlobalTheoremAlternateBase} and the Ergodic Theorem for the first equality. The conclusion now follows from Proposition~\ref{Pro : Equiv}.
\end{proof}

\section{Isomorphism between greedy and lazy $\B$-transformations}\label{Section : IsoGreedyLazy}

In this section, we show that 
\begin{equation}
\label{Eq : IsoGreedyLazy}
	\phi_{\B}\colon 
	\bigcup_{i=0}^{p-1} \big(\{i\} \times [0,x_{\B^{(i)}}) \big)	
	\to 	\bigcup_{i=0}^{p-1} \big(\{i\} \times (0,x_{\B^{(i)}}] \big),\ 
	(i,x)\mapsto \big(i, x_{\B^{(i)}}-x\big)
\end{equation}
defines an isomorphism between the greedy $\B$-transformation and the lazy $\B$-transfor\-mation. 

We consider the $\sigma$-algebra 
\[
	\mathcal{L}_{\B}
	=\Bigg\{
	\bigcup_{i=0}^{p-1} 
	(\{i\}\times B_i)\colon 
	\forall i\in\Int,\ 
	B_i\in\mathcal{B}\big((0, x_{\B^{(i)}}] \big)
	\Bigg\}
\]
on $\bigcup_{i=0}^{p-1} \big(\{i\} \times \big(0,x_{\B^{(i)}}\big] \big)$. 

\begin{theorem}
\label{Thm : IsomorphismLazy}
The map $\phi_{\B}$ is an isomorphism between the dynamical systems 
$\big(\bigcup_{i=0}^{p-1} 
	\big(\{i\} \times [0, x_{\B^{(i)}}) \big),
	\mathcal{T}_{\B}, \mu_{\B}, T_{\B}\big)$
and $\big(\bigcup_{i=0}^{p-1} 
	\big(\{i\} \times (0, x_{\B^{(i)}}] \big),
	\mathcal{L}_{\B},\mu_{\B} \circ \phi_{\B}^{-1},L_{\B}\big)$.
\end{theorem}

\begin{proof}
Clearly, $\phi_{\B}$ is a bijective map. Hence, we only have to show that $\phi_{\B}\circ T_{\B}=L_{\B}\circ \phi_{\B}$. Let $(i,x)\in \bigcup_{i=0}^{p-1} \big(\{i\} \times [0,x_{\B^{(i)}}) \big)$. First, suppose that $x\in[0,1)$. Then
\[
	\phi_{\B}\circ T_{\B}(i,x)
	=\big((i+1)\bmod p,
	x_{\B^{(i+1)}}-\beta_ix+\floor{\beta_i x}
	\big)
\]
and
\[
	L_{\B}\circ \phi_{\B}(i,x)
	=\big((i+1)\bmod p,
	\beta_i(x_{\B^{(i)}}-x)
	-\ceil{\beta_i(x_{\B^{(i)}}-x)-x_{\B^{(i+1)}}}		\big).
\]
Second, suppose that $x\in[1,x_{\B^{(i)}})$. Then
\[
	\phi_{\B}\circ T_{\B}(i,x)
	=\big((i+1)\bmod p,
	x_{\B^{(i+1)}}-\beta_ix+\floor{\beta_i}-1
	\big)
\]
and
\[
	L_{\B}\circ \phi_{\B}(i,x)
	=\big((i+1)\bmod p,
	\beta_i(x_{\B^{(i)}}-x)
	\big).
\]
In both cases, we easily get that 
$\phi_{\B}\circ T_{\B}(i,x)=	L_{\B}\circ \phi_{\B}(i,x)
$ by using~\eqref{Eq : EqualitiesXB}. 
\end{proof}

Thanks to Theorem~\ref{Thm : IsomorphismLazy}, we obtain an analogue of Theorem~\ref{Thm : GlobalTheoremAlternateBaseExtended} for the lazy $\B$-transfor\-mation. 

\begin{theorem}
The measure $\mu_{\B} \circ \phi_{\B}^{-1}$ is the unique $L_{\B}$-invariant probability measure on $\mathcal{L}_{\B}$ that is absolutely continuous with respect to $\lambda_p\circ\phi_{\B}^{-1}$. Furthermore, $\mu_{\B} \circ \phi_{\B}^{-1}$ is equivalent to $\lambda_p\circ\phi_{\B}^{-1}$ on $\mathcal{L}_{\B}$ and the dynamical system $\big(\bigcup_{i=0}^{p-1} 
	\big(\{i\} \times (0,x_{\B^{(i)}}] \big),
	\mathcal{L}_{\B}, \mu_{\B} \circ \phi_{\B}^{-1}, L_{\B}\big)$ is ergodic and has entropy $\frac{1}{p}\log(\beta_{p-1}\cdots\beta_0)$.
\end{theorem}

Similarly, we have an analogue of Theorem~\ref{Thm : GlobalTheoremAlternateBase} for the lazy $\B$-transformation, by considering the $\sigma$-algebra
\[
	\mathcal{L}_{\B}'
	=\Bigg\{
	\bigcup_{i=0}^{p-1} 
	(\{i\}\times B_i)\colon 
	\forall i\in\Int,\ 
	B_i\in\mathcal{B}\big((x_{\B^{(i)}}-1, x_{\B^{(i)}}] \big)
	\Bigg\}.
\]

\begin{theorem}
The measure $\mu_{\B} \circ \phi_{\B}^{-1}$ is the unique $L_{\B}$-invariant probability measure on $\mathcal{L}_{\B}'$ that is absolutely continuous with respect to $\lambda_p\circ\phi_{\B}^{-1}$. Furthermore, $\mu_{\B} \circ \phi_{\B}^{-1}$ is equivalent to $\lambda_p\circ\phi_{\B}^{-1}$ on $\mathcal{L}_{\B}'$ and the dynamical system $\big(\bigcup_{i=0}^{p-1} 
	\big(\{i\} \times (x_{\B^{(i)}}-1,x_{\B^{(i)}}] \big),
	\mathcal{L}_{\B}', \mu_{\B} \circ \phi_{\B}^{-1}, L_{\B}\big)$ is ergodic and has entropy $\frac{1}{p}\log(\beta_{p-1}\cdots\beta_0)$.
\end{theorem}

\begin{remark}
\label{Rem : chiffres-lazy}
We deduce from Theorem~\ref{Thm : IsomorphismLazy} that if the greedy $\B$-expansion of a real number $x\in [0,x_{\B})$ is $a_0a_1a_2\cdots$, then the lazy $\B$-expansion of $x_{\B}-x$ is $(\ceil{\beta_0}-1-a_0)(\ceil{\beta_1}-1-a_1)(\ceil{\beta_2}-1-a_2)\cdots$. 
\end{remark}

\section{Isomorphism with the $\B$-shift}
\label{Section : Shift}

The aim of this section is to generalize the isomorphism between the greedy $\beta$-transfor\-mation and the $\beta$-shift to the framework of alternate bases. We start by providing some background of the real base case.

Let $D_\beta$ denote the set of all greedy $\beta$-expansions of real numbers in the interval $[0,1)$. The \emph{$\beta$-shift} is the set $S_\beta$ defined as the topological closure of $D_\beta$ with respect to the prefix distance of infinite words. For an alphabet $A$, we let $\mathcal{C}_{A}$ denote the $\sigma$-algebra generated by the \emph{cylinders}
\[
	C_A(a_0,\ldots,a_{\ell-1})=\{w \in A^{\N} \colon w[0]=a_0,\ldots,w[\ell-1]=a_{\ell-1}\}
\]
for all $\ell\in \N$ and $a_0,\ldots,a_{\ell-1}\in A$, where the notation $w[k]$ designates the letter at position $k$ in the infinite word $w$, and we call 
\[
	\sigma_A \colon A^{\N} \to A^{\N},\ 
	a_0a_1a_2\cdots\mapsto a_1a_2a_3\cdots
\] 
the \emph{shift operator} over $A$. If no confusion is possible, we simply write $\sigma$ instead of $\sigma_A$. Then the map $\psi_{\beta}\colon [0,1) \to S_\beta,\ x \mapsto d_\beta(x)$ defines an isomorphism between the dynamical systems $( [0,1),\mathcal{B}([0,1)),\mu_\beta,T_\beta)$ and $( S_\beta, \{C\cap S_\beta\colon C\in \mathcal{C}_{A_\beta}\}, \mu_{\beta}\circ \psi_{\beta}^{-1},\sigma_{|S_\beta})$ where $A_\beta$ denote the alphabet of digits $[\![0,\ceil{\beta}-1]\!]$.

Now, let us extend the previous notation to the framework of alternate bases. Let $A_{\B}$ denote the alphabet $[\![0,\max\limits_{i\in\Int}\ceil{\beta_i}-1]\!]$, let $D_{\B}$ denote the subset of $A_{\B}^\N$ made of all greedy $\B$-expansions of real numbers in $[0,1)$ and let $S_{\B}$ denote the topological closure of $D_{\B}$ with respect to the prefix distance of infinite words: 
\[
	D_{\B}=\{d_{\B}(x) \colon x \in [0,1)\} 
	\quad \text{and} \quad 
	S_{\B}=\overline{D_{\B}}.
\]
The following lemma was proved in~\cite{Charlier&Cisternino:2020}.

\begin{lemma}
\label{Lem : SbetaShift}
For all $n\in\N$, if $w\in S_{\B^{(n)}}$ then $\sigma(w)\in S_{\B^{(n+1)}}$.
\end{lemma}

Consider the $\sigma$-algebra 
\[
	\mathcal{G}_{\B}
	=\Bigg\{\bigcup_{i=0}^{p-1}\big(\{i\}\times (C_i\cap S_{\B^{(i)}})\big)
	\colon C_i\in \mathcal{C}_{A_{\B}}\Bigg\}
\] 
on $\bigcup_{i=0}^{p-1} (\{i\}\times S_{\B^{(i)}})$. We define 
\begin{align*}
	&\sigma_p\colon 
	\bigcup_{i=0}^{p-1} (\{i\}\times S_{\B^{(i)}}) \to 	\bigcup_{i=0}^{p-1} (\{i\}\times S_{\B^{(i)}}),\ 
	(i,w)\mapsto ((i+1)\bmod p,\sigma(w))\\
	&\psi_{\B}\colon 
	\Int\times[0,1)\to 	\bigcup_{i=0}^{p-1} (\{i\}\times S_{\B^{(i)}})	,\ 
	(i,x)\mapsto (i,d_{\B^{(i)}}(x)).
\end{align*} 
Note that the transformation $\sigma_p$ is well defined by Lemma~\ref{Lem : SbetaShift}. 

\begin{theorem}
\label{Thm : AlternateBaseIsomorphism}
The map $\psi_{\B}$ defines an isomorphism between the dynamical systems 
\[
	\big(\Int \times [0,1),\mathcal{T}_p,\mu_{\B},T_{\B}\big)
	\quad\text{and}\quad
	\Bigg(\bigcup_{i=0}^{p-1} (\{i\}\times S_{\B^{(i)}}),\mathcal{G}_{\B},\mu_{\B}\circ \psi_{\B}^{-1},\sigma_p\Bigg).
\]
\end{theorem}

\begin{proof}
It is easily seen that $\psi_{\B}\circ T_{\B}=\sigma_p\circ \psi_{\B}$ and that $\psi_{\B}$ is injective. 
\end{proof}

However, since $\psi_{\B}$ is not surjective, it does not define a topological isomorphism.

\begin{remark}
In view of Theorem~\ref{Thm : AlternateBaseIsomorphism}, the set $\bigcup_{i=0}^{p-1} (\{i\}\times S_{\B^{(i)}})$ can be seen as the \emph{$\B$-shift}, that is, the generalization of the $\beta$-shift to alternate bases. However, in the previous work \cite{Charlier&Cisternino:2020}, what we called the $\B$-shift is the union $\bigcup_{i=0}^{p-1} S_{\B^{(i)}}$. This definition was motivated by the  following combinatorial result : the set $\bigcup_{i=0}^{p-1} S_{\B^{(i)}}$ is sofic if and only if for every $i\in\Int$, the quasi-greedy $\B^{(i)}$-representation of $1$ is ultimately periodic. In summary, we can say that there are two ways to extend the notion of $\beta$-shift to alternate bases $\B$, depending on the way we look at it: either as a dynamical object or as a combinatorial object.
\end{remark}

Thanks to Theorem~\ref{Thm : AlternateBaseIsomorphism}, we obtain an analogue of Theorem~\ref{Thm :  GlobalTheoremAlternateBase} for the transformation $\sigma_p$.

\begin{theorem}
The measure $\rho_{\B}$ is the unique $\sigma_p$-invariant probability measure on $\mathcal{G}_{\B}$ that is absolutely continuous with respect to $\lambda_p\circ\psi_{\B}^{-1}$. Furthermore, $\rho_{\B}$ is equivalent to $\lambda_p\circ\psi_{\B}^{-1}$ on $\mathcal{G}_{\B}$ and the dynamical system $	\big(\bigcup_{i=0}^{p-1} (\{i\}\times S_{\B^{(i)}}),\mathcal{G}_{\B},\rho_{\B},\sigma_p\big)$ is ergodic and has entropy $\frac{1}{p} \log(\beta_{p-1}\cdots\beta_0)$.
\end{theorem}

\begin{remark}
Let $D_{\B}'$ denote the subset of $A_{\B}^\N$ made of all lazy $\B$-expansions of real numbers in $(x_{\B}-1, x_{\B}]$ and let $S_{\B}'$ denote the topological closure of $D_{\B}'$ with respect to the prefix distance of infinite words. From Remark~\ref{Rem : chiffres-lazy}, it is easily seen that 
\[
	\theta_{\B}\colon \bigcup_{i=0}^{p-1} (\{i\}\times S_{\B^{(i)}}) \to \bigcup_{i=0}^{p-1} (\{i\}\times S_{\B^{(i)}}'),\ 
	(i,a_0a_1\cdots)\mapsto (i,(\ceil{\beta_i}-1-a_0)(\ceil{\beta_{i+1}}-1-a_2)\cdots)
\]
defines a isomorphism from $	\big(\bigcup_{i=0}^{p-1} (\{i\}\times S_{\B^{(i)}}),\mathcal{G}_{\B},\rho_{\B},\sigma_p\big)$ to $\big(\bigcup_{i=0}^{p-1} (\{i\}\times S_{\B^{(i)}}'),\mathcal{G}_{\B}',\rho_{\B}\circ\theta_{\B}^{-1},\sigma_p'\big)$ where 
\begin{align*}
	\mathcal{G}_{\B}'
	&=\Bigg\{\bigcup_{i=0}^{p-1}\big(\{i\}\times (C_i\cap S_{\B^{(i)}}')\big)
	\colon C_i\in \mathcal{C}_{A_{\B}}\Bigg\}	\\
	\sigma_p' &\colon 
	\bigcup_{i=0}^{p-1} (\{i\}\times S_{\B^{(i)}}') \to 	\bigcup_{i=0}^{p-1} (\{i\}\times S_{\B^{(i)}}'),\ 
	(i,w)\mapsto ((i+1)\bmod p,\sigma(w)).
\end{align*}
We then deduce from Theorem~\ref{Thm : IsomorphismLazy}  and~\ref{Thm : AlternateBaseIsomorphism} that $\theta_{\B}\circ \psi_{\B}\circ\phi_{\B}^{-1}$ is an isomorphism from $\big(\bigcup_{i=0}^{p-1} 
	\big(\{i\} \times (x_{\B^{(i)}}-1, x_{\B^{(i)}}] \big),
	\mathcal{L}_{\B},\mu_{\B} \circ \phi_{\B}^{-1},L_{\B}\big)$ to $\big(\bigcup_{i=0}^{p-1} (\{i\}\times S_{\B^{(i)}}'),\mathcal{G}_{\B}',\rho_{\B}\circ\theta_{\B}^{-1},\sigma_p'\big)$ where here $\phi_{\B}$ denoted the restricted map
\[
	\phi_{\B}\colon 
	\bigcup_{i=0}^{p-1} \big(\{i\} \times [0,1) \big)	
	\to 	\bigcup_{i=0}^{p-1} \big(\{i\} \times (x_{\B^{(i)}}-1,x_{\B^{(i)}}] \big),\ 
	(i,x)\mapsto \big(i, x_{\B^{(i)}}-x\big).
\]
It is easy to check that, as expected, that for all $(i,x)\in \bigcup_{i=0}^{p-1} 
	\big(\{i\} \times (x_{\B^{(i)}}-1, x_{\B^{(i)}}]$,  we have
$\theta_{\B}\circ \psi_{\B}\circ\phi_{\B}^{-1}(i,x)=(i,\ell_{\B^{(i)}}(x))$ where $\ell_{\B}(x)$ denoted the lazy $\B$-expansion of $x$.
\end{remark}

\section{$\B$-expansions and $(\beta_{p-1}\cdots\beta_0,\Delta_{\B})$-expansions}
\label{Section : ComparisonBexpansionProduct}

By rewriting Equality~\eqref{Eq : valueAlternatBase} from Section~\ref{Section : AlternateBases} as
\begin{align}
\label{Eq : BexpansionProduct}
x&= \frac{\beta_{p-1}\cdots\beta_1 a_0 +  \beta_{p-1}\cdots\beta_2 a_1 + \cdots +a_{p-1}}{\beta_{p-1}\cdots\beta_0} \\
& + \frac{\beta_{p-1}\cdots\beta_1 a_p 
+ \beta_{p-1}\cdots\beta_1 a_{p+1} + \cdots +a_{2p-1}}{(\beta_{p-1}\cdots\beta_0)^2}\nonumber \\
& +	\cdots\nonumber
\end{align} 
we can see the greedy and lazy $\B$-expansions of real numbers as $(\beta_{p-1}\cdots\beta_0)$-representations over the digit set 
\[
	\Delta_{\B}=\Bigg\{\sum_{i=0}^{p-1}\beta_{p-1}\cdots\beta_{i+1} c_i
	\colon \forall i\in \Int,\ c_i\in [\![0,\ceil{\beta_i}-1]\!]\Bigg\}.
\]
In this section, we examine some cases where by considering the greedy (resp.\ lazy) $\B$-expansion and rewriting it as~\eqref{Eq : BexpansionProduct}, the obtained representation is the greedy (resp.\ lazy) $(\beta_{p-1}\cdots\beta_0,\Delta_{\B})$-expansion. We first recall the formalism of $\beta$-expansions of real numbers over a general digit set~\cite{Pedicini:2005}. 

\subsection{Real base expansions over general digit sets}
Consider an arbitrary finite set $\Delta=\{d_0,d_1,\ldots,d_m\}\subset \R$ where $0=d_0<d_1<\cdots<d_m$. Then a \emph{$(\beta,\Delta)$-representation} of a real number $x$ in the interval $[0,\frac{d_m}{\beta-1})$ is an infinite sequence $a_0a_1a_2\cdots$ over $\Delta$ such that $x=\sum_{n=0}^\infty \frac{a_n}{\beta^{n+1}}$. Such a set $\Delta$ is called an \emph{allowable digit set for $\beta$} if 
\begin{equation}
\label{Eq : Pedicini}
	\max_{k\in[\![0,m-1]\!]}(d_{k+1}-d_k)\le \frac{d_m}{\beta-1}.
\end{equation} 
In this case, the \emph{greedy $(\beta,\Delta)$-expansion} of a real number $x\in [0,\frac{d_m}{\beta-1})$ is defined recursively as follows: if the first $N$ digits of the greedy $(\beta,\Delta)$-expansion of $x$ are given by $a_0,\ldots,a_{N-1}$, then the next digit $a_N$ is the greatest element in $\Delta$ such that 
\[
	\sum_{n=0}^N\frac{a_n}{\beta^{n+1}} \le x.
\]
The greedy $(\beta,\Delta)$-expansion can also be obtained by iterating the \emph{greedy $(\beta,\Delta)$-transfor\-mation} 
\[
	T_{\beta,\Delta} \colon 
	[0,\tfrac{d_m}{\beta-1}) \to [0,\tfrac{d_m}{\beta-1}),\ 
	x \mapsto
	\begin{cases}										\beta x-d_k 	&\text{if }x\in[\tfrac{d_k}{\beta}, \tfrac{d_{k+1}}{\beta}),\ k\in[\![0,m-1]\!]\\			\beta x-d_m 	&\text{if }x\in[\tfrac{d_m}{\beta},\tfrac{d_m}{\beta-1})
\end{cases}
\]
as follows: for all $n\in\N$, $a_n$ is the greatest digit $d$ in $\Delta$ such that $\frac{d}{\beta}\le T^n_{\beta,\Delta}(x)$~\cite{Dajani&Kalle:2007}.

\begin{example}
\label{Ex : DigitSetGreedy}
Consider the digit set $\Delta=\{0,1,\varphi+\frac{1}{\varphi},\varphi^2\}$. It is easily checked that $\Delta$ is an allowable digit set for $\varphi$.
The greedy $(\varphi,\Delta)$-transformation
\[
	T_{\varphi,\Delta}\colon 
	[0,\tfrac{\varphi^2}{\varphi-1})\to [0,\tfrac{\varphi^2}{\varphi-1}),\
	x\mapsto
\begin{cases}
\varphi x  		& \text{if }x\in [0,\tfrac{1}{\varphi})\\
\varphi x - 1	& \text{if }x\in[\tfrac{1}{\varphi}, 1+\tfrac{1}{\varphi^2})\\
\varphi x - (\varphi+\tfrac{1}{\varphi}) & \text{if } x\in[1+\tfrac{1}{\varphi^2},\varphi)\\
\varphi x - \varphi^2 & \text{if } x \in[\varphi, \tfrac{\varphi^2}{\varphi-1})
\end{cases}
\]
is depicted in Figure~\ref{Fig : TvarphiDelta}.
\begin{figure}[htb]
\centering
\begin{tikzpicture}[scale=0.3]
\draw[line width=0.0mm ] (0,0) rectangle (16.9443,16.9443); 
\draw[line width=0.0mm ] (0,0)  to (16.9443,16.9443); 
\draw (-0.1,-0.2) node[left]{$0$};
\draw (2.47214,0) node[below]{$\smath{\frac{1}{\varphi}}$};
\draw (5.52786-0.5,0) node[below]{$\smath{1+\frac{1}{\varphi^2}}$};
\draw (6.9,0) node[below]{$\smath{\varphi}$};
\draw (16.9443,0) node[below]{$\smath{\frac{\varphi^2}{\varphi-1}}$};
\draw (0,3.9) node[left]{$\smath{1}$};
\draw (0,5.1) node[left]{$\smath{\frac{2}{\varphi}}$};
\draw (0,1.52786)  node[left]{$\smath{\frac{1}{\varphi^2}}$};
\draw (0,16.9443) node[left]{$\smath{\frac{\varphi^2}{\varphi-1}}$};
\draw[line width=0.4mm,blue] (0,0)  to (2.47214,4);
\draw[line width=0.4mm,blue] (2.47214,0) to (5.52786, 4.94427);
\draw[line width=0.4mm,blue] (5.52786,0) to (6.4721, 1.52786);
\draw[line width=0.4mm,blue] (6.4721,0) to (16.9443, 16.9443);
\draw[dotted] (2.47214,0) to (2.47214,16.9443); 
\draw[dotted] (5.52786,0) to (5.52786,16.9443); 
\draw[dotted] (6.4721,0) to (6.4721,16.9443); 
\draw[dotted] (0,4) to (16.9443,4); 
\draw[dotted] (0,4.94427) to (16.9443,4.94427); 
\draw[dotted] (0,1.52786) to (16.9443,1.52786); 
\end{tikzpicture}
\caption{The transformation $T_{\varphi,\Delta}$ for  $\Delta=\{0,1,\frac{\varphi+1}{\varphi},\varphi^2\}$.}
\label{Fig : TvarphiDelta}
\end{figure}
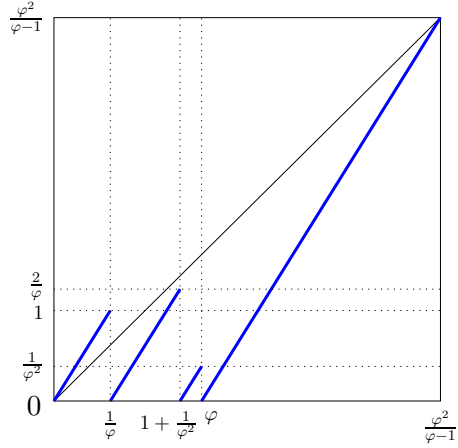
\end{example}

Similarly, if $\Delta$ is an allowable digit set for $\beta$, then the \emph{lazy $(\beta,\Delta)$-expansion} of a real number $x\in (0,\frac{d_m}{\beta-1}]$ is defined recursively as follows: if the first $N$ digits of the lazy $(\beta,\Delta)$-expansion of $x$ are given by $a_0,\ldots,a_{N-1}$, then the next digit $a_N$ is the least element in $\Delta$ such that 
\[
	\sum_{n=0}^N\frac{a_n}{\beta^{n+1}} 	+\sum_{n=N+1}^{\infty}\frac{d_m}{\beta^{n+1}}
	\ge x.
\]
The \emph{lazy $(\beta,\Delta)$-transformation}
\[
	L_{\beta,\Delta} \colon 
	(0,\tfrac{d_m}{\beta-1}] \to (0,\tfrac{d_m}{\beta-1}],\ 
	x \mapsto
	\begin{cases}		
	\beta x 		&\text{if }x\in (0,\tfrac{d_m}{\beta-1}{-}\tfrac{d_m}{\beta}]\\
	\beta x-d_k 	&\text{if }x\in (\tfrac{d_m}{\beta-1}{-}\tfrac{d_m-d_{k-1}}{\beta},\tfrac{d_m}{\beta-1}{-}\tfrac{d_m-d_{k}}{\beta}],\ k\in[\![1,m]\!]
\end{cases}
\]
can be used to obtain the digits of the lazy $(\beta,\Delta)$-expansions: for all $n\in\N$, $a_n$ is the least digit $d$ in $\Delta$ such that $\frac{d}{\beta} +\sum_{k=1}^{\infty}\frac{d_m}{\beta^{k+1}} \ge L^n_{\beta,\Delta}(x)$~\cite{Dajani&Kalle:2007}.

In~\cite{Dajani&Kalle:2007}, it is shown that if $\Delta$ is an allowable digit set for $\beta$ then so is the set $\widetilde{\Delta}:=\{0,d_m{-}d_{m-1},\ldots, d_m{-}d_1,d_m\}$ and
\[
	\phi_{\beta,\Delta} \colon 
	[0,\tfrac{d_m}{\beta-1}) \to (0,\tfrac{d_m}{\beta-1}],\ 
	x \mapsto \tfrac{d_m}{\beta-1}-x
\]
is a bicontinuous bijection satisfying $L_{\beta,\widetilde{\Delta}}\, \circ \,\phi_{\beta,\Delta} = \phi_{\beta,\Delta}\, \circ\,T_{\beta,\Delta}$.

\begin{example}\label{Ex : DigitSetLazy}
Consider the digit set $\widetilde{\Delta}$ where $\Delta$ is the digit set from Example~\ref{Ex : DigitSetGreedy}. We get $\widetilde{\Delta}=\{0,1-\frac{1}{\varphi},\varphi,\varphi^2\}$.
The lazy $(\varphi,\widetilde{\Delta})$-transformation
\[
	L_{\varphi,\widetilde{\Delta}}\colon 
	(0,\tfrac{\varphi^2}{\varphi-1}]\to (0,\tfrac{\varphi^2}{\varphi-1}],\
	x\mapsto
\begin{cases}
\varphi x  		& \text{if }x\in (0,\tfrac{\varphi}{\varphi-1}]\\
\varphi x - (1-\frac{1}{\varphi})	& \text{if }x\in(\tfrac{\varphi}{\varphi-1}, \tfrac{\varphi+3}{\varphi}]\\
\varphi x - \varphi & \text{if } x\in(\tfrac{\varphi+3}{\varphi},\tfrac{2\varphi-1}{\varphi-1}]\\
\varphi x - \varphi^2 & \text{if } x \in(\tfrac{2\varphi-1}{\varphi-1}, \tfrac{\varphi^2}{\varphi-1}]
\end{cases}
\]
is depicted in Figure~\ref{Fig : LvarphiDelta}. It is conjugate to the greedy $(\varphi,\Delta)$-transformation $T_{\varphi,\Delta}$ by $\phi_{\varphi,\Delta} \colon [0,\tfrac{\varphi^2}{\varphi-1}) \to (0,\tfrac{\varphi^2}{\varphi-1}],\ x \mapsto \tfrac{\varphi^2}{\varphi-1}-x$.
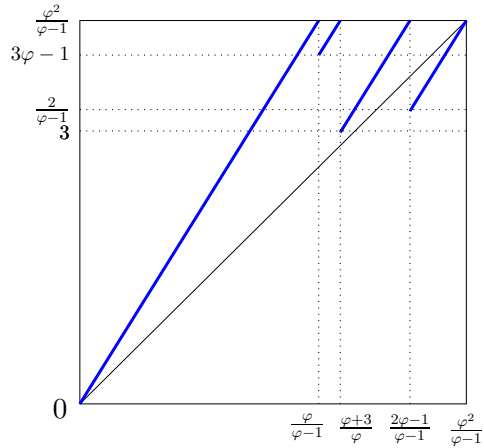
\begin{figure}[htb]
\centering
\begin{tikzpicture}[scale=0.3]
\draw[line width=0.0mm ] (0,0) rectangle (16.9443,16.9443); 
\draw[line width=0.0mm ] (0,0)  to (16.9443,16.9443); 
\draw (-0.1,-0.2) node[left]{$0$};
\draw ( 10.4721-0.5,0) node[below]{$\smath{\tfrac{\varphi}{\varphi-1}}$};
\draw(11.4164+0.7,0)  node[below]{$\smath{\tfrac{\varphi+3}{\varphi}}$};
\draw(14.4721,0)node[below]{$\smath{\tfrac{2\varphi-1}{\varphi-1}}$};
\draw (16.9443,0) node[below]{$\smath{\frac{\varphi^2}{\varphi-1}}$};
\draw[line width=0.4mm,blue] (0,0)  to (10.4721,16.9443);
\draw[line width=0.4mm,blue] (10.4721,15.4164) to (11.4164, 16.9443);
\draw[line width=0.4mm,blue] (11.4164,12) to (14.4721, 16.9443);
\draw[line width=0.4mm,blue] (14.4721, 12.9443) to (16.9443,16.9443); 
\draw[dotted] (10.4721,0) to (10.4721,16.9443); 
\draw[dotted] (11.4164,0) to (11.4164,16.9443); 
\draw[dotted] (14.4721,0) to (14.4721,16.9443);
\draw[dotted] (0,15.4164) to (16.9443,15.4164); 
\draw (0,15.4164) node[left]{$\smath{3\varphi-1}$};
\draw[dotted] (0,12.0647) to (16.9443,12.0647); 
\draw (0,12) node[left]{$\smath{3}$};
\draw[dotted] (0,13.009) to (16.9443,13.009); 
\draw (0,12) node[left]{$\smath{3}$};
\draw (0,12.9443) node[left]{$\smath{\frac{2}{\varphi-1}}$};
\draw (0,16.9443) node[left]{$\smath{\frac{\varphi^2}{\varphi-1}}$};
\end{tikzpicture}
\caption{The transformation $L_{\varphi,\widetilde{\Delta}}$ for  $\Delta=\{0,1,\varphi+\frac{1}{\varphi},\varphi^2\}$.}
\label{Fig : LvarphiDelta}
\end{figure}
\end{example}

\subsection{Comparison between $\B$-expansions and $(\beta_{p-1}\cdots\beta_0,\Delta_{\B})$-expansions}

The digit set $\Delta_{\B}$ has cardinality at most $\prod_{i=0}^{p-1} \ceil{\beta_i}$ and can be rewritten $\Delta_{\B}=\mathrm{im}(f_{\B})$ where 
\[
	f_{\B} \colon 
	\prod_{i=0}^{p-1}\ [\![0,\ceil{\beta_i}-1]\!]\to\R,\ 
	(c_0,\ldots, c_{p-1}) 
	\mapsto \sum_{i=0}^{p-1}\beta_{p-1}\cdots \beta_{i+1}c_i.
\]
Note that $f_{\B}$ is not injective in general. Let us write $\Delta_{\B}=\{d_0,d_1\ldots,d_m\}$ with $d_0<d_1<\cdots < d_m$. We have $d_0=f_{\B}(0,\ldots,0)=0$, $d_1=f_{\B}(0,\ldots,0,1)=1$ and $d_m=f_{\B}(\ceil{\beta_0}-1,\ldots,\ceil{\beta_{p-1}}-1)$. In what follows, we suppose that $\prod_{i=0}^{p-1}\ [\![0,\ceil{\beta_i}-1]\!]$ is equipped with the lexicographic order:  $(c_0,\ldots,c_{p-1})<_\lex(c'_0,\ldots,c'_{p-1})$ if there exists $i\in\Int$ such that $c_0=c_0',\ldots,c_{i-1}=c_{i-1}'$ and $c_i<c_i'$.

\begin{lemma}
\label{Lem : AllowableDiditSet}
The set $\Delta_{\B}$ is an allowable digit set for $\beta_{p-1}\cdots \beta_0$.
\end{lemma}

\begin{proof}
We need to check Condition~\eqref{Eq : Pedicini}. We have $d_0=0$ and
\[
	d_m=f_{\B}(\ceil{\beta_0}-1,\ldots,\ceil{\beta_{p-1}}-1)
	\ge \sum_{i=0}^{p-1}\beta_{p-1}\cdots \beta_{i+1}(\beta_i-1)
	=\beta_{p-1}\cdots\beta_0-1,
\]
Therefore, it suffices to show that for all $k\in [\![0,m-1 ]\!]$, $d_{k+1}-d_k\le1$. Thus, we only have to show that $f(c_0',\ldots,c_{p-1}')-f(c_0,\ldots,c_{p-1})\le 1$ where $(c_0,\ldots,c_{p-1})$ and $(c'_0,\ldots,c'_{p-1})$ are lexicographically consecutive elements of $\prod_{i=0}^{p-1}\ [\![0,\ceil{\beta_i}-1]\!]$. For such $p$-tuples, there exists $j\in \Int$ such that $c_0=c'_0,\ldots,c_{j-1}=c'_{j-1}$, $c_j=c'_j-1$, $c_{j+1}=\ceil{\beta_{j+1}}-1,\ldots,c_{p-1}=\ceil{\beta_{p-1}}-1$ and $c'_{j+1}=\cdots=c'_{p-1}=0$. Then
\begin{align*}
f(c_0',\ldots,c_{p-1}')-f(c_0,\ldots,c_{p-1})	
	&=\beta_{p-1}\cdots\beta_{j+1}
	-\sum_{i=j+1}^{p-1}\beta_{p-1}\cdots\beta_{i+1}(\ceil{\beta_{i}}-1)\\
	&\le \beta_{p-1}\cdots\beta_{j+2}
	-\sum_{i=j+2}^{p-1}\beta_{p-1}\cdots\beta_{i+1}(\ceil{\beta_i}-1)\\	
	& \ \,  \vdots\\
	\medskip
	&\le \beta_{p-1}-(\ceil{\beta_{p-1}}-1)\\
&\le 1.
\end{align*}
\end{proof}

Since $x_{\B}=\frac{d_m}{\beta_{p-1}\cdots\beta_0-1}$, it follows from Lemma~\ref{Lem : AllowableDiditSet} that every point in $[0,x_{\B})$ admits a greedy $(\beta_{p-1}\cdots\beta_0,\Delta_{\B})$-expansion.

Let us restate Proposition~\ref{Pro : LexMax} when $n$ equals $p$ in terms of the map $f_{\B}$.

\begin{lemma}
\label{Lem : LexMaxP}
For all $x\in[0,x_{\B})$, we have
\[
	\pi_2 \circ T_{\B}^p \circ \delta_0 (x)
	=\beta_{p-1}\cdots\beta_0 x-f_{\B}(c)
\]
where $c$ is the lexicographically greatest $p$-tuple in $\prod_{i=0}^{p-1}\ [\![0,\ceil{\beta_i}-1]\!]$  such that $\frac{ f_{\B}(c)}{\beta_{p-1}\cdots\beta_0} \le x$. 
\end{lemma}

\begin{proposition}
\label{Pro : plusgrand}
For all $x\in [0,x_{\B})$, we have $T_{\beta_{p-1}\cdots\beta_0,\Delta_{\B}}(x) \le \pi_2 \circ T_{\B}^p \circ \delta_0(x)$ and $L_{\beta_{p-1}\cdots\beta_0,\Delta_{\B}}(x) \ge \pi_2 \circ L_{\B}^p \circ \delta_0(x)$. 
\end{proposition}

\begin{proof}
Let $x\in[0,x_{\B})$. On the one hand, $T_{\beta_{p-1}\cdots\beta_0,\Delta_{\B}}(x)=\beta_{p-1}\cdots\beta_0 x -d$ where $d$ is the greatest digit in $\Delta_{\B}$ such that $\frac{d}{\beta_{p-1}\cdots\beta_0} \le x$. On the other hand, by Lemma~\ref{Lem : LexMaxP}, $\pi_2 \circ T_{\B}^p \circ \delta_0(x)=\beta_{p-1}\cdots\beta_0 x -f_{\B}(c)$ where $c$ is the greatest $p$-tuple in $\prod_{i=0}^{p-1}\ [\![0,\ceil{\beta_i}-1]\!]$ such that $\frac{f_{\B}(c)}{\beta_{p-1}\cdots\beta_0} \le x$. By definition of $d$, we get $d\ge f_{\B}(c)$. Therefore, we obtain that $T_{\beta_{p-1}\cdots\beta_0,\Delta_{\B}}(x) \le \pi_2 \circ T_{\B}^p \circ \delta_0(x)$. The inequality $L_{\beta_{p-1}\cdots\beta_0,\Delta_{\B}}(x) \ge \pi_2 \circ L_{\B}^p \circ \delta_0(x)$ then follows from Theorem~\ref{Thm : IsomorphismLazy}.
\end{proof}

In what follows, we provide some conditions under which the inequalities of Proposition~\ref{Pro : plusgrand} happen to be equalities.

\begin{proposition}
\label{Pro : GreedyCoincideIIFLazyCoincide}
The transformations $T_{\beta_{p-1}\cdots\beta_0,\Delta_{\B}}$ and ${\pi_2 \circ T_{\B}^p \circ \delta_0}_{\big|[0,x_{\B})}$ coincide if and only if 
the transformations $L_{\beta_{p-1}\cdots\beta_0,\Delta_{\B}}$ and ${\pi_2 \circ L_{\B}^p \circ \delta_0}_{\big|(0,x_{\B}]}$ do.
\end{proposition}

\begin{proof}
We only show the forward direction, the backward direction being similar. Suppose that $T_{\beta_{p-1}\cdots\beta_0,\Delta_{\B}}={\pi_2 \circ T_{\B}^p \circ \delta_0}_{\big|[0,x_{\B})}$ and let $x\in(0,x_{\B}]$. Since $x_{\B}=\frac{d_m}{\beta_{p-1}\cdots\beta_0-1}$ and $\Delta_{\B}= \widetilde{\Delta_{\B}}$, we successively obtain that
\begin{align*}
	L_{\beta_{p-1}\cdots\beta_0,\Delta_{\B}}(x)
	&=L_{\beta_{p-1}\cdots\beta_0,\Delta_{\B}}
	\circ\phi_{\beta_{p-1}\cdots\beta_0,\Delta_{\B}}(x_{\B}-x)\\
	&=\phi_{\beta_{p-1}\cdots\beta_0,\Delta_{\B}}	\circ T_{\beta_{p-1}\cdots\beta_0,\Delta_{\B}}
(x_{\B}-x)\\
	&=\phi_{\beta_{p-1}\cdots\beta_0,\Delta_{\B}}	\circ\pi_2\circ T_{\B}^p\circ\delta_0(x_{\B}-x)\\
	&=\pi_2\circ\phi_{\B}\circ T_{\B}^p\circ\delta_0(x_{\B}-x)\\
	&=\pi_2\circ L_{\B}^p\circ\phi_{\B}\circ \delta_0(x_{\B}-x)\\
	&={\pi_2 \circ L_{\B}^p \circ \delta_0}(x).
\end{align*} 
\end{proof}

The next result provides us with a sufficient condition under which the transformations $T_{\beta_{p-1}\cdots\beta_0,\Delta_{\B}}$ and ${\pi_2 \circ T_{\B}^p \circ \delta_0}_{\big|[0,x_{\B})}$ coincide. Here, the non-decreasingness of the map $f_{\B}$ refers to the lexicographic order: for all $c,c'\in \prod_{i=0}^{p-1}\ [\![0,\ceil{\beta_i}-1]\!]$, $c<_{\lex}c'\implies f_{\B}(c) \le f_{\B}(c')$.

\begin{theorem}
\label{Thm : fIncreasingEqual}
If the map $f_{\B}$ is non-decreasing then $T_{\beta_{p-1}\cdots\beta_0,\Delta_{\B}}={\pi_2 \circ T_{\B}^p \circ \delta_0}_{\big|[0,x_{\B})}$. \end{theorem}

\begin{proof}
We keep the same notation as in the proof of Proposition~\ref{Pro : plusgrand}. Let $c'\in \prod_{i=0}^{p-1}\ [\![0,\ceil{\beta_i}-1]\!]$ such that $d=f_{\B}(c')$. By definition of $c$, we get $c\ge_{\lex} c'$. Now, if $f_{\B}$ is non-decreasing then $f_{\B}(c)\ge f_{\B}(c')=d$. Hence the conclusion.
\end{proof}

The following example shows that considering the length-$p$ alternate base $\B=(\beta,\ldots,\beta)$ with $p\in \N_{\ge 3}$, it may happen that $T_{\beta^p,\Delta_{\B}}$ differs from ${\pi_2 \circ T_{\B}^p \circ \delta_0}_{\big|[0,x_{\B})}$. This result was already proved in~\cite{Dajani&Vries&Komornik&Loreti:2012}. 

\begin{example}
Consider the alternate base $\B=(\varphi^2,\varphi^2,\varphi^2)$.  Then $\Delta_{\B}=\{\varphi^4 c_0 + \varphi^2 c_1 +c_2\colon c_0,c_1,c_2 \in \{0,1,2\}\}$. In~\cite[Proposition 2.1]{Dajani&Vries&Komornik&Loreti:2012}, it is proved that $T_{\beta^n,\Delta_{\B}}=T_{\beta}^n$ for all $n\in \N$ if and only if $f_{\B}$ is non-decreasing. Since $f_{\B}(0,2,2)=2\varphi^2 +2>\varphi^4=f_{\B}(1,0,0)$, the tranformations $T_{\varphi^6,\Delta_{\B}}$ and ${\pi_2 \circ T_{\B}^3 \circ \delta_0}_{\big|[0,x_{\B})}$ differ by~\cite[Proposition 2.1]{Dajani&Vries&Komornik&Loreti:2012}. 
\end{example}

Whenever $f_{\B}$ is not non-decreasing, the transformations $T_{\beta_{p-1}\cdots\beta_0,\Delta_{\B}}$ and ${\pi_2 \circ T_{\B}^p \circ \delta_0}_{\big|[0,x_{\B})}$ can either coincide or not. The following two examples illustrate both cases. In particular, Example~\ref{Ex : counterex} shows that the sufficient condition given in Theorem~\ref{Thm : fIncreasingEqual} is not necessary.

\begin{example}
Consider the alternate base $\B=(\varphi,\varphi,\sqrt{5})$. Then $\Delta_{\B}=\{\sqrt{5}\varphi c_0 + \sqrt{5}c_1 +c_2\colon c_0,c_1 \in \{0,1\},\ c_2 \in \{0,1,2\} \}$. However, $f_{\B}(0,1,2)=\sqrt{5}+2\simeq 4.23$ and $f_{\B}(1,0,0)=\sqrt{5}\varphi\simeq  3.61$. 
It can be easily check that there exists $x \in [0,x_{\B})$ such that $T_{\sqrt{5}\varphi^2,\Delta_{\B}}(x) \ne \pi_2 \circ T_{\B}^3 \circ \delta_0(x)$. For example, we can compute $T_{\sqrt{5}\varphi^2,\Delta_{\B}}(0.75) \simeq 0.15$ and $\pi_2 \circ T_{\B}^3 \circ \delta_0(0.75) \simeq 0.77$. The transformations $T_{\sqrt{5}\varphi^2,\Delta_{\B}}$ and ${\pi_2 \circ T_{\B}^3 \circ \delta_0}_{\big|[0,x_{\B})}$ are depicted in Figure~\ref{Fig : PhiPhiSqrt5}, where the red lines show the images of the interval $\big[\frac{\sqrt{5}+2}{\sqrt{5}\varphi^2},\frac{\sqrt{5}\varphi+1}{\sqrt{5}\varphi^2} \big)\simeq[0.72,0.78)$, that is where the two transformations differ. Similarly, the transformations $L_{\sqrt{5}\varphi^2,\Delta_{\B}}$ and ${\pi_2 \circ L_{\B}^3 \circ \delta_0}_{\big|(0,x_{\B}]}$ are depicted in Figure~\ref{Fig : PhiPhiSqrt5Lazy}. As illustrated in red, the two transformations differ on the interval $\phi_{\sqrt{5}\varphi^2,\Delta_{\B}}\Big(\big[\frac{\sqrt{5}+2}{\sqrt{5}\varphi^2},\frac{\sqrt{5}\varphi+1}{\sqrt{5}\varphi^2} \big)\Big)\simeq(0.82,0.89]$.
\begin{figure}[htb]
\centering
\begin{minipage}{.5\linewidth}
\begin{tikzpicture}[scale=0.6]
\draw[line width=0.0mm ] (0,0) rectangle (9.7082,9.7082); 
\draw[dotted,line width=0.3mm ] (6,0) to (6,6); 
\draw[dotted,line width=0.3mm ] (0,6) to (6,6); 
\draw[dotted,line width=0.15mm  ] (6,6) to (6,9.7082); 
\draw[dotted ,line width=0.15mm ] (6,6)to (9.7082,6); 
\draw[line width=0.0mm ] (0,0)  to (9.7082,9.7082); 
\draw (-0.1,-0.2) node[left]{$0$};
\draw (9.7082,0) node[below]{$x_{\B}$};
\draw (0,9.7082) node[left]{$x_{\B}$};
\draw (6,0) node[below]{$1$};
\draw (0,6) node[left]{$1$};
\draw[line width=0.4mm,blue] (0,0) to (1.02492,6);
\draw[line width=0.4mm,blue] (1.02492,0) to ( 2.04984,6);
\draw[line width=0.4mm,blue] (2.04984,0) to ( 2.2918,1.41641);
\draw[line width=0.4mm,blue] (2.2918,0) to ( 3.31672,6);
\draw[line width=0.4mm,blue] (3.31672,0) to ( 3.7082,2.2918);
\draw[line width=0.4mm,red] (3.7082,0) to ( 4.34164,3.7082);
\draw[line width=0.4mm,red] (4.3416,0) to (4.73313 ,2.2918);
\draw[line width=0.4mm,blue] (4.73313,0) to (5.75805 ,6);
\draw[line width=0.4mm,blue] (5.75805,0) to (6,1.41641);
\draw[line width=0.4mm,blue] (6,0) to (7.02492,6);
\draw[line width=0.4mm,blue] (7.02492,0) to (8.04984,6);
\draw[line width=0.4mm,blue] (8.04984,0) to (9.7082,9.7082);
\end{tikzpicture}
\end{minipage}%
\begin{minipage}{.5\linewidth}
\begin{tikzpicture}[scale=0.6]
\draw[line width=0.0mm ] (0,0) rectangle (9.7082,9.7082); 
\draw[dotted,line width=0.3mm ] (6,0) to (6,6); 
\draw[dotted,line width=0.3mm ] (0,6) to (6,6); 
\draw[dotted,line width=0.15mm  ] (6,6) to (6,9.7082); 
\draw[dotted ,line width=0.15mm ] (6,6)to (9.7082,6); 
\draw[line width=0.0mm ] (0,0)  to (9.7082,9.7082); 
\draw (-0.1,-0.2) node[left]{$0$};
\draw (9.7082,0) node[below]{$x_{\B}$};
\draw (0,9.7082) node[left]{$x_{\B}$};
\draw (6,0) node[below]{$1$};
\draw (0,6) node[left]{$1$};
\draw[line width=0.4mm,blue] (0,0) to (1.02492,6);
\draw[line width=0.4mm,blue] (1.02492,0) to ( 2.04984,6);
\draw[line width=0.4mm,blue] (2.04984,0) to ( 2.2918,1.41641);
\draw[line width=0.4mm,blue] (2.2918,0) to ( 3.31672,6);
\draw[line width=0.4mm,blue] (3.31672,0) to ( 3.7082,2.2918);
\draw[line width=0.4mm,red] (3.7082,0) to ( 4.73313,6);
\draw[line width=0.4mm,blue] (4.73313,0) to (5.75805 ,6);
\draw[line width=0.4mm,blue] (5.75805,0) to (6,1.41641);
\draw[line width=0.4mm,blue] (6,0) to (7.02492,6);
\draw[line width=0.4mm,blue] (7.02492,0) to (8.04984,6);
\draw[line width=0.4mm,blue] (8.04984,0) to (9.7082,9.7082);
\end{tikzpicture}
\end{minipage}
\caption{The transformations $T_{\sqrt{5}\varphi^2,\Delta_{\B}}$ (left) and ${\pi_2 \circ T_{\B}^3 \circ \delta_0}_{\big|[0,x_{\B})}$ (right) with $\B=(\varphi,\varphi,\sqrt{5})$.}
\label{Fig : PhiPhiSqrt5}
\end{figure}
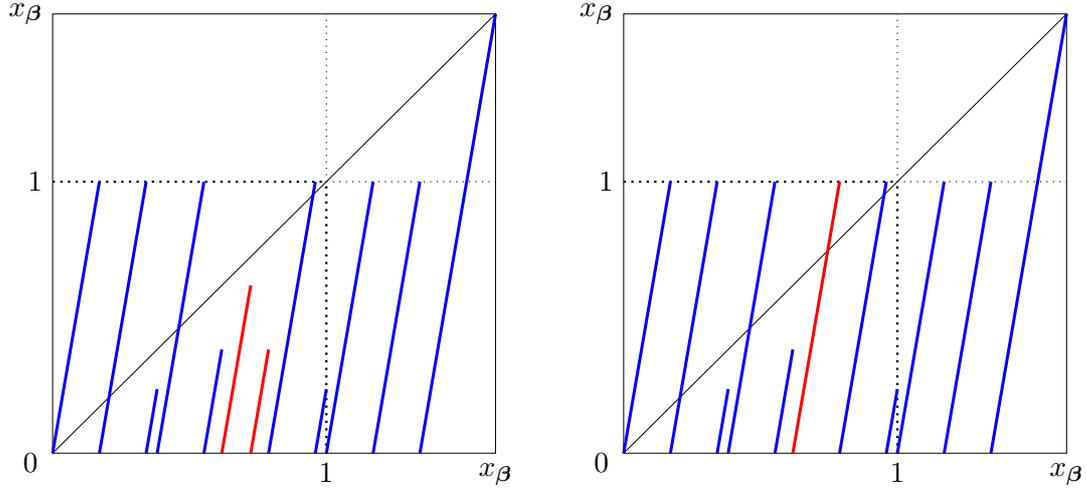
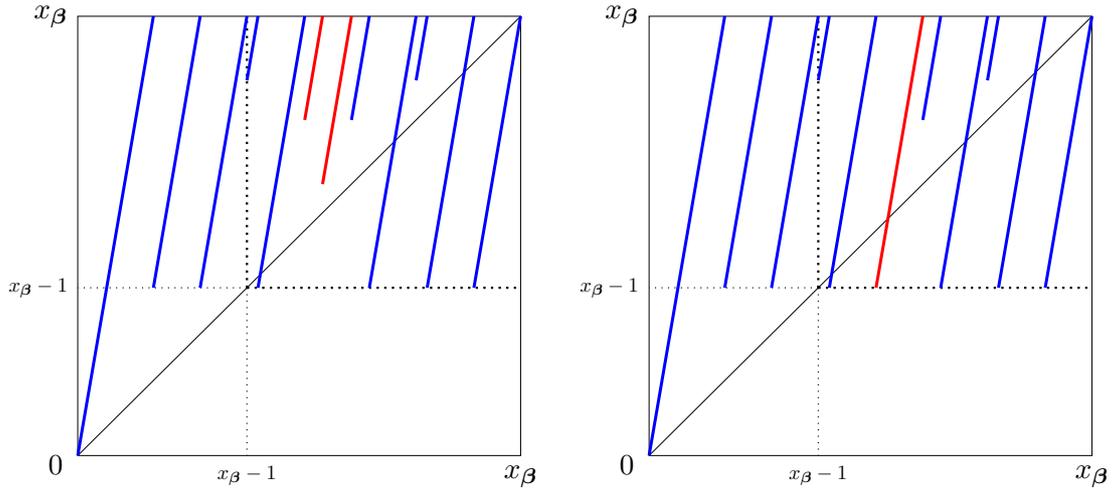
\begin{figure}[htb]
\centering
\begin{minipage}{.5\linewidth}
\begin{tikzpicture}[scale=0.6]
\draw[line width=0.0mm ] (0,0) rectangle (9.7082,9.7082); 
\draw[dotted,line width=0.3mm ] (9.7082-6,9.7082-6) to (9.7082,9.7082-6); 
\draw[dotted,line width=0.3mm ] (9.7082-6,9.7082-6) to (9.7082-6,9.7082); 
\draw[dotted,line width=0.15mm ]  (9.7082-6,9.7082-6)to  (9.7082-6,0); 
\draw[dotted,line width=0.15mm ]  (9.7082-6,9.7082-6)to  (0,9.7082-6); 
\draw[line width=0.0mm ] (0,0)  to (9.7082,9.7082); 
\draw (-0.1,-0.2) node[left]{$0$};
\draw (9.7082,0) node[below]{$x_{\B}$};
\draw (0,9.7082) node[left]{$x_{\B}$};
\draw (9.7082-6,0) node[below]{$\smath{x_{\B}-1}$};
\draw (0,9.7082-6) node[left]{$\smath{x_{\B}-1}$};
\draw[line width=0.4mm,blue] (0,0) to (1.65836,9.7082);
\draw[line width=0.4mm,blue] (1.65836,3.7082) to (2.68328,9.7082);
\draw[line width=0.4mm,blue] (2.68328,3.7082) to (3.7082,9.7082);
\draw[line width=0.4mm,blue] (3.7082,8.2918) to (3.95016,9.7082);
\draw[line width=0.4mm,blue] (3.95016,3.7082) to (4.97508,9.7082);
\draw[line width=0.4mm,red] (4.97508,7.41641) to (5.36656,9.7082);
\draw[line width=0.4mm,red] (5.36656,6) to (6,9.7082);
\draw[line width=0.4mm,blue] (6,7.41641) to (6.39149,9.7082);
\draw[line width=0.4mm,blue] (6.39149,3.7082) to (7.41641,9.7082);
\draw[line width=0.4mm,blue] (7.41641,8.2918) to (7.65836,9.7082);
\draw[line width=0.4mm,blue] (7.65836,3.7082) to (8.68328,9.7082);
\draw[line width=0.4mm,blue] (8.68328,3.7082) to (9.7082,9.7082);
\end{tikzpicture}
\end{minipage}%
\begin{minipage}{.5\linewidth}
\begin{tikzpicture}[scale=0.6]
\draw[line width=0.0mm ] (0,0) rectangle (9.7082,9.7082); 
\draw[dotted,line width=0.3mm ] (9.7082-6,9.7082-6) to (9.7082,9.7082-6); 
\draw[dotted,line width=0.3mm ] (9.7082-6,9.7082-6) to (9.7082-6,9.7082); 
\draw[dotted,line width=0.15mm ]  (9.7082-6,9.7082-6)to  (9.7082-6,0); 
\draw[dotted,line width=0.15mm ]  (9.7082-6,9.7082-6)to  (0,9.7082-6); 
\draw[line width=0.0mm ] (0,0)  to (9.7082,9.7082); 
\draw (-0.1,-0.2) node[left]{$0$};
\draw (9.7082,0) node[below]{$x_{\B}$};
\draw (0,9.7082) node[left]{$x_{\B}$};
\draw (9.7082-6,0) node[below]{$\smath{x_{\B}-1}$};
\draw (0,9.7082-6) node[left]{$\smath{x_{\B}-1}$};
\draw[line width=0.4mm,blue] (0,0) to (1.65836,9.7082);
\draw[line width=0.4mm,blue] (1.65836,3.7082) to (2.68328,9.7082);
\draw[line width=0.4mm,blue] (2.68328,3.7082) to (3.7082,9.7082);
\draw[line width=0.4mm,blue] (3.7082,8.2918) to (3.95016,9.7082);
\draw[line width=0.4mm,blue] (3.95016,3.7082) to (4.97508,9.7082);
\draw[line width=0.4mm,red] (4.97508,3.7082) to (6,9.7082);
\draw[line width=0.4mm,blue] (6,7.41641) to (6.39149,9.7082);
\draw[line width=0.4mm,blue] (6.39149,3.7082) to (7.41641,9.7082);
\draw[line width=0.4mm,blue] (7.41641,8.2918) to (7.65836,9.7082);
\draw[line width=0.4mm,blue] (7.65836,3.7082) to (8.68328,9.7082);
\draw[line width=0.4mm,blue] (8.68328,3.7082) to (9.7082,9.7082);
\end{tikzpicture}
\end{minipage}
\caption{The transformations $L_{\sqrt{5}\varphi^2,\Delta_{\B}}$ (left) and ${\pi_2 \circ L_{\B}^3 \circ \delta_0}_{\big|[0,x_{\B})}$ (right) with $\B=(\varphi,\varphi,\sqrt{5})$.}
\label{Fig : PhiPhiSqrt5Lazy}
\end{figure}
\end{example} 

\begin{example}
\label{Ex : counterex}
Consider the alternate base $\B=(\frac{3}{2},\frac{3}{2},4)$. We have $\Delta_{\B}=[\![0,13]\!]$. The map $f_{\B}$ is not non-decreasing since we have $f_{\B}(0,1,3)=7$ and $f_{\B}(1,0,0)=6$. However, $T_{9,\Delta_{\B}}={\pi_2 \circ T_{\B}^3 \circ \delta_0}_{\big|[0,x_{\B})}$ and $L_{9,\Delta_{\B}}={\pi_2 \circ L_{\B}^3 \circ \delta_0}_{\big|[0,x_{\B})}$. The transformation $T_{9,\Delta_{\B}}$ is depicted in Figure~\ref{Fig : counterex}.
\begin{figure}
\begin{tikzpicture}[scale=3.5]
\draw[line width=0.0mm ] (0,0) rectangle (1.625,1.625); 
\draw[dotted,line width=0.3mm ] (1,0) to (1,1); 
\draw[dotted,line width=0.3mm ] (0,1) to (1,1); 
\draw[dotted,line width=0.15mm  ] (1,1) to (1,1.625); 
\draw[dotted ,line width=0.15mm ] (1,1)to (1.625,1); 
\draw[line width=0.0mm ] (0,0)  to (1.625,1.625); 
\draw (-0.01,-0.02) node[left]{$0$};
\draw (1.625,0) node[below]{$x_{\B}$};
\draw (0,1.625) node[left]{$x_{\B}$};
\draw (1,0) node[below]{$1$};
\draw (0,1) node[left]{$1$};
\draw[line width=0.4mm,blue] (0,0) to (0.111111,1);
\draw[line width=0.4mm,blue] (0.111111,0) to (0.111111*2,1);
\draw[line width=0.4mm,blue] (0.111111*2,0) to (0.111111*3,1);
\draw[line width=0.4mm,blue] (0.111111*2,0) to (0.111111*3,1);
\draw[line width=0.4mm,blue] (0.111111*3,0) to (0.111111*4,1);
\draw[line width=0.4mm,blue] (0.111111*4,0) to (0.111111*5,1);
\draw[line width=0.4mm,blue] (0.111111*5,0) to (0.111111*6,1);
\draw[line width=0.4mm,blue] (0.111111*6,0) to (0.111111*7,1);
\draw[line width=0.4mm,blue] (0.111111*7,0) to (0.111111*8,1);
\draw[line width=0.4mm,blue] (0.111111*8,0) to (0.111111*9,1);
\draw[line width=0.4mm,blue] (0.111111*9,0) to (0.111111*10,1);
\draw[line width=0.4mm,blue] (0.111111*10,0) to (0.111111*11,1);
\draw[line width=0.4mm,blue] (0.111111*11,0) to (0.111111*12,1);
\draw[line width=0.4mm,blue] (0.111111*12,0) to (0.111111*13,1);
\draw[line width=0.4mm,blue] (0.111111*13,0) to (1.625,1.625);
\end{tikzpicture}
\caption{The transformations $T_{9,\Delta_{\B}}$ where $\B=(\frac{3}{2},\frac{3}{2},4)$.}
\label{Fig : counterex}
\end{figure}
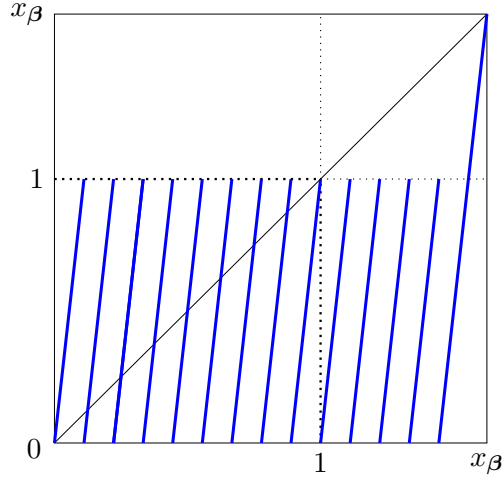

\end{example}

The next example illustrates that it may happen that the transformations  $T_{\beta_{p-1}\cdots\beta_0,\Delta_{\B}}$ and ${\pi_2 \circ T_{\B}^p \circ \delta_0}_{\big|[0,x_{\B})}$ indeed coincide on $[0,1)$ but not on $[0,x_{\B})$.

\begin{example}
Consider the alternate base $\B=(\frac{\sqrt{5}}{2},\frac{\sqrt{6}}{2},\frac{\sqrt{7}}{2})$. Then $f_{\B}(0,1,1)>f_{\B}(1,0,0)$ and it can be checked that the maps  $T_{\frac{\sqrt{210}}{8},\Delta_{\B}}$ and ${\pi_2 \circ T_{\B}^3 \circ \delta_0}_{\big|[0,x_{\B})}$ differ on the interval $\big[\frac{f_{\B}(0,1,1)}{\beta_2\beta_1\beta_0},\frac{f_{\B}(1,0,1)}{\beta_2\beta_1\beta_0}\big)\simeq [1.28,1.44)$. However, the two maps coincide on $[0,1)$.
\end{example}

Finally, we provide a necessary and sufficient condition for the map $f_{\B}$ to be non-decreasing. 

\begin{proposition}
\label{Pro : CaractNonDecreasing}
The map $f_{\B}$ is non-decreasing if and only if for all $j\in[\![1,p-2]\!]$,
\begin{equation}
\label{Eq : NonDecreasing}
	\sum_{i=j}^{p-1}\beta_{p-1}\cdots\beta_{i+1}(\ceil{\beta_i}-1)\le \beta_{p-1}\cdots\beta_j.
\end{equation}
\end{proposition}

\begin{proof}
If the map $f_{\B}$ is non-decreasing then for all $j\in[\![1,p-2]\!]$, 
\begin{align*}
\sum_{i=j}^{p-1}\beta_{p-1}\cdots\beta_{i+1}(\ceil{\beta_i}-1)
&= f_{\B}(0,\ldots,0,0,\ceil{\beta_j}-1,\ldots,\ceil{\beta_{p-1}}-1)\\ 
&\le f_{\B}(0,\ldots,0,1,0,\ldots,0)\\
&=\beta_{p-1}\cdots\beta_j.
\end{align*}

Conversely, suppose that~\eqref{Eq : NonDecreasing} holds for all $j\in[\![1,p-2]\!]$ and that $(c_0,\ldots,c_{p-1})$ and $(c_0',\ldots,c_{p-1}')$ are $p$-tuples in $\prod_{i=0}^{p-1}\ [\![0,\ceil{\beta_i}-1]\!]$ such that $(c_0,\ldots,c_{p-1}) <_{\lex}(c_0',\ldots,c_{p-1}')$. Then there exists $j\in\Int$ such that $c_0=c_0',\ldots,c_{j-1}=c_{j-1}'$ and $c_j\le c_j'-1$. We get
\begin{align*}
	f_{\B}(c_0,\ldots,c_{p-1}) 
	&\le 
	\sum_{i=0}^j\beta_{p-1}\cdots\beta_{i+1}c_i' 
	- \beta_{p-1}\cdots\beta_{j+1}
	+\sum_{i=j+1}^{p-1}\beta_{p-1}\cdots\beta_{i+1}(\ceil{\beta_i}-1)\\ 
	&\le 
	\sum_{i=0}^j\beta_{p-1}\cdots\beta_{i+1}c_i'\\		&\le f_{\B}(c_0',\ldots,c_{p-1}').
\end{align*}
\end{proof}

\begin{corollary}
\label{Coro : Length2Equal}
If $p=2$ then $T_{\beta_1\beta_0,\Delta_{\B}}={\pi_2 \circ T_{\B}^2 \circ \delta_0}_{\big|[0,x_{\B})}$. 
In particular, ${T_{\beta_1\beta_0,\Delta_{\B}}}_{\big|[0,1)}=T_{\beta_1} \circ T_{\beta_0}$.
\end{corollary}

\begin{proof}
This follows from Theorem~\ref{Thm : fIncreasingEqual} and Proposition~\ref{Pro : CaractNonDecreasing}.
\end{proof}

\begin{example}
\label{Ex : TBsquared13}
Consider once more the alternate base $\B=(\frac{1+\sqrt{13}}{2}, \frac{5+\sqrt{13}}{6})$ from Example~\ref{Ex : TbetaRacine13}. Then $\Delta_{\B}=\{ 0,1,\beta_1,\beta_1+1, 2\beta_1, 2\beta_1+1\}$ and $x_{\B}= \frac{2\beta_1+1}{\beta_1\beta_0-1}= \frac{5 +7 \sqrt{13}}{18}$. The transformations ${\pi_2 \circ T_{\B}^2 \circ \delta_0}_{\big|[0,x_{\B})}$ and ${\pi_2 \circ L_{\B}^2 \circ \delta_0}_{\big|(0,x_{\B}]}$ are depicted in Figure~\ref{Fig : TBsquared13}. By Corollary~\ref{Coro : Length2Equal}, they coincides with $T_{\beta_1 \beta_0,\Delta_{\B}}$ and $L_{\beta_1 \beta_0,\Delta_{\B}}$ respectively.
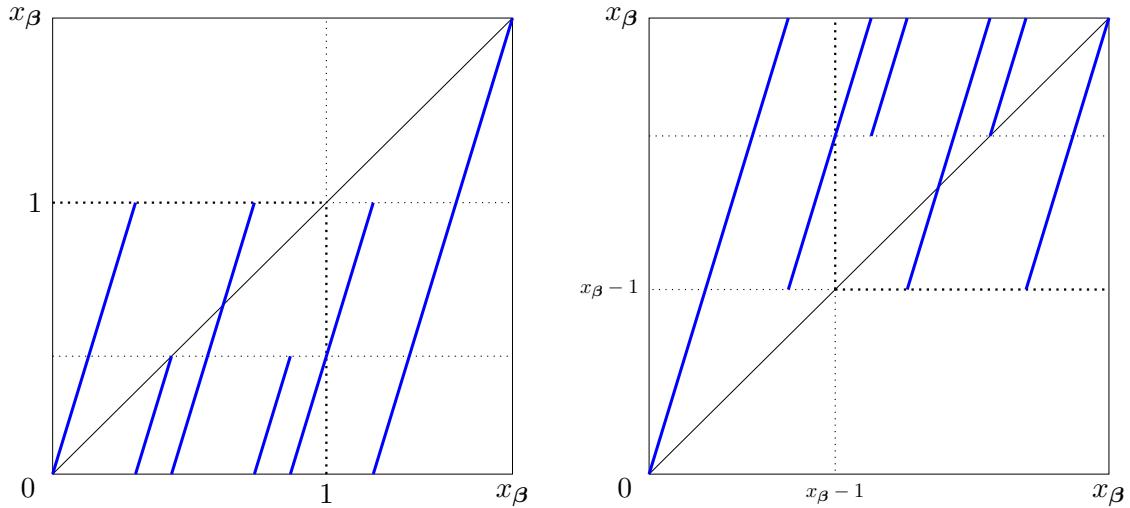
\begin{figure}[htb]
\centering
\begin{minipage}{.5\linewidth}
\begin{tikzpicture}[scale=0.9]
\draw[line width=0.0mm ] (0,0) rectangle (6.71975,6.71975); 
\draw[dotted,line width=0.3mm ] (4,0) to (4,4); 
\draw[dotted,line width=0.3mm ] (0,4) to (4,4); 
\draw[dotted,line width=0.15mm ] (4,4) to (4,6.71975); 
\draw[dotted,line width=0.15mm ] (4,4) to (6.71975,4); 
\draw[dotted,line width=0.15mm ] (0,1.73703) to (6.71975,1.73703); 
\draw[line width=0.0mm ] (0,0) to (6.71975,6.71975); 
\draw (-0.1,-0.2) node[left]{$0$};
\draw (4,0) node[below]{$1$};
\draw (0,4) node[left]{$1$};
\draw (6.71975,0) node[below]{$x_{\B}$};
\draw (0,6.71975) node[left]{$x_{\B}$};
\draw[line width=0.4mm,blue] (0,0) to (1.2111,4); 
\draw[line width=0.4mm,blue] (1.2111,0) to (1.73703,1.73703); 
\draw[line width=0.4mm,blue] (1.73703,0) to (2.94814,4); 
\draw[line width=0.4mm,blue] (2.94814,0) to (3.47407,1.73703); 
\draw[line width=0.4mm,blue] (3.47407,0) to (4.68517,4); 
\draw[line width=0.4mm,blue] (4.68517,0) to (6.71975,6.71975); 
\end{tikzpicture}
\end{minipage}%
\begin{minipage}{.5\linewidth}
\begin{tikzpicture}[scale=0.9]
\draw[line width=0.0mm ] (0,0) rectangle (6.71975,6.71975); 
\draw[dotted,line width=0.3mm ] (6.71975-4,6.71975-4)to (6.71975,6.71975-4); 
\draw[dotted,line width=0.3mm ] (6.71975-4,6.71975-4) to (6.71975-4,6.71975); 
\draw[dotted,line width=0.15mm ] (6.71975-4,6.71975-4)  to (0,6.71975-4); 
\draw[dotted,line width=0.15mm ] (6.71975-4,6.71975-4)  to (6.71975-4,0);
\draw[dotted,line width=0.15mm ] (0,4.98271)  to (6.71975,4.98271); 
\draw[line width=0.0mm ] (0,0) to (6.71975,6.71975); 
\draw (-0.1,-0.2) node[left]{$0$};
\draw (6.71975-4,0) node[below]{$\smath{x_{\B}-1}$};
\draw (0,6.71975-4) node[left]{$\smath{x_{\B}-1}$};
\draw (6.71975,0) node[below]{$x_{\B}$};
\draw (0,6.71975) node[left]{$x_{\B}$};
\draw[line width=0.4mm,blue] (0,0) to (2.03458,6.71975); 
\draw[line width=0.4mm,blue] (2.03458,2.71975) to (3.24568,6.71975); 
\draw[line width=0.4mm,blue] (3.24568,4.98271) to (3.77161,6.71975); 
\draw[line width=0.4mm,blue] (3.77161,2.71975) to (4.98271,6.71975); 
\draw[line width=0.4mm,blue] (4.98271,4.98271) to (5.50864,6.71975); 
\draw[line width=0.4mm,blue] (5.50864,2.71975) to (6.71975,6.71975); 
\end{tikzpicture}
\end{minipage}%
\caption{The transformations ${\pi_2 \circ T_{\B}^2 \circ \delta_0}_{\big|[0,x_{\B})}$ (left) and ${\pi_2 \circ L_{\B}^2 \circ \delta_0}_{\big|(0,x_{\B}]}$ (right) for $\B=(\tfrac{1+\sqrt{13}}{2}, \tfrac{5+\sqrt{13}}{6})$.}
\label{Fig : TBsquared13}
\end{figure}
\end{example}

\section{Acknowledgment}
We thank Julien Leroy for suggesting Lemma~\ref{Lem : Julien}, which allowed us to simplify several proofs. Célia Cisternino is supported by the FNRS Research Fellow grant 1.A.564.19F.

\bibliographystyle{abbrv}
\bibliography{Charlier-Cisternino-Dajani}

\end{document}